\documentclass[psamsfonts]{amsart}
\usepackage{amsmath,amstext,amssymb,amsopn,amsthm}
\usepackage{amsmath,amssymb,amsthm}
\usepackage[mathscr]{eucal}
\usepackage{bbm}
\usepackage{color}
\usepackage{tikz}
\usetikzlibrary{arrows}
\usepackage[unicode,bookmarks,colorlinks]{hyperref}
\hypersetup{
    linkcolor=brickred,
}

\definecolor{mahogany}{cmyk}{0, 0.77, 0.87, 0}
\definecolor{salmon}{cmyk}{0, 0.53, 0.38, 0}
\definecolor{melon}{cmyk}{0, 0.46, 0.50, 0}
\definecolor{yellowgreen}{cmyk}{0.44, 0, 0.74, 0}
\definecolor{brickred}{cmyk}{0, 0.89, 0.94, 0.28}
\definecolor{OliveGreen}{cmyk}{0.64, 0, 0.95, 0.40}
\definecolor{RawSienna}{cmyk}{0, 0.72, 1.0, 0.45}
\definecolor{ZurichRed}{rgb}{1, 0, 0} 

\numberwithin{equation}{section}

\newtheorem{thm}{Theorem}[section]        

\newtheorem{lem}{Lemma}[section]
\newtheorem{prop}{Proposition}[section]
\newtheorem{rmk}{Remark}[section]
\newtheorem{definition}{Definition}[section]

\newenvironment{•}{•}{•}
\newcommand{\R}{\mathbb{R}}                  
\newcommand{\Rd}{\R^d}               
\newcommand{\ioRd}{\int_{\Rd}}              
\newcommand{\set}[1]{ \left\{#1\right\} }
\newcommand{\mysum}[3]{\sum\limits_{#1=#2}^{#3}}          

\newcommand{\E}[1]{\mathbb{E}\left[#1\right]}

\newcommand{\eid}{\,\,{\buildrel \mathcal{D} \over =}\,\,}   

\newcommand{\abs}[1]{\left|#1\right|}

\newcommand{\tgo}{t\downarrow 0}
\newcommand{\F} {(-\Delta)^{ \frac{\alpha}{2} } }
\newcommand{\Prob} {\mathbb{P} }
\newcommand{\palp}{p_{t}^{(\alpha)}}

\newcommand{\Halp}{\mathbb{H}_{\Omega}^{(\alpha)}(t)}
\newcommand{\Hset}[2]{\mathbb{H}_{#1,#2}^{(\alpha)}(t)}
\newcommand{\ele}{\ell_{\alpha}(t)} 
\newcommand{\inalp}{ \frac{1}{\alpha}} 
\newcommand{\ld}{\abs{\Omega}}
\newcommand{\myexp}[1]{\exp\left(#1\right)}
\newcommand{\pthesis}[1]{\left(#1\right)}
\newcommand{\alpr}{0<\alpha<2}
\newcommand{\pint}[1]{\left[#1\right]}
\newcommand{\bd}{\partial\Omega}

\newcommand{\md}{\rho_{\Omega}(x)}
\newcommand{\myH}{\mathbb{H}}

\newcommand{\myP}{\mathcal{P}}
\usepackage[margin=1.25in]{geometry}
\begin{document}

\title[ Heat content ]{Heat content for stable processes in domains of $\R^d$}
\author{Luis Acu\~na Valverde}\thanks{Supported in part by NSF Grant
\#0603701-DMS, Rodrigo  Ba\~nuelos, PI}
\address{Department of Mathematics, Universidad de Costa Rica, San Jos\'e, Costa Rica.}
\email{luis.acunavalverde@ucr.ac.cr/ guillemp22@yahoo.com}
\maketitle

\begin{abstract}
This paper studies   the small time behavior of the heat content for  rotationally invariant $\alpha$--stable processes, $0<\alpha \leq 2$, in domains of $\R^d$.  Unlike the asymptotics for the heat trace,  the behavior of the heat content differs depending on the range of $\alpha$ according to   $0<\alpha<1$, $\alpha=1$ and $1<\alpha\leq 2$.  
\end{abstract}

\section{Introduction}
Let ${\bf B}=\set{B_t}_{t\geq0}$ be a $d$--dimensional Brownian motion in a probability
space $(\mathcal{N}_1,\{\mathcal{F}_t^{\bf B}\}_{t\geq 0},\Prob_{\bf B}^x)$ and let  ${\bf S}=\set{S_t}_{t\geq 0}$ be an $\alpha/2$--subordinator started at zero in the probability space  $(\mathcal{N}_2,\mathcal{G},\mathbb{P}_{\bf S})$ with $0<\alpha\leq2$. We recall that ${\bf S}=\set{S_t}_{t\geq 0}$  is a  one-dimensional L\'{e}vy--process taking values on $[0,\infty)$ with increasing sample paths. That is, for every $u,t\geq 0$, the increment $S_{t+u}-S_u$ is independent of the process $\set{S_{\tau},0\leq \tau\leq u }$ and has the same law as $S_t$.  Moreover, if $0\leq t\leq u$, we have  $0\leq S_t\leq S_u$ almost everywhere.  When $\alpha=2$, we adopt the convention $S_t=t$.  We will consider both processes ${\bf B}$  and ${\bf S }$ on the product space 
$\mathcal{N}=\mathcal{N}_1 \times \mathcal{N}_2$.  In addition, we set $\mathcal{F}_t=\mathcal{F}_t^{\bf B}
\times\mathcal{G}$ and $\mathbb{P}^x=\mathbb{P}_{\bf B}^x \times \mathbb{P}_{\bf S}.$ 
Hence, ${\bf B}$ is a $d$-dimensional $\mathcal{F}_t$--Brownian motion and
${\bf S}$ is an $\alpha/2$--subordinator independent of ${\bf B}$ when they are regarded as stochastic
 processes defined over $(\mathcal{N}, \Prob^x)$ (we refer the reader to  \cite
{Song2} for further details). 

Consider the stochastic process ${\bf X}=\set{X_t}_{t\geq 0}$ in the probability
space $(\mathcal{N},\{\mathcal{F}_t\}_{t\geq 0},\Prob ^x)$ defined as $X_t(w_1,w_2)=B_{2S_t(w_2)}(w_1)$ for every $(w_1,w_2)\in \mathcal{N}$. Then, due to the independence between ${\bf B}$ and ${\bf S}$ stated in the first paragraph when they are regarded as stochastic
 processes defined over $(\mathcal{N}, \Prob^x)$ and Theorem 1.3.25  in \cite{appleb}, we obtain that  ${\bf X}$ is  a rotationally invariant $\alpha$-stable L\'evy process in $\R^d$ whose Fourier transform of   its  transition densities, denoted throughout the paper by $\palp(x,y)=\palp(x-y)$, satisfy
\begin{equation}\label{CharF.Proc}
\exp\pthesis{-t\abs{\xi}^{\alpha}}=\ioRd 
\exp\pthesis{-\dot{\iota}<y,\xi>}\palp(y)dy,
\end{equation}
for all $t>0$, $\xi\in\Rd$. Before providing some basic properties about the heat kernels $\palp(x,y)$, we introduce the following standard notation. $\mathbb{E}^x$ and $\Prob^x$ will denote the expectation 
and  probability of any  process started at $x$, respectively. Also  for  simplicity, 
we will connote $\Prob=\Prob^{0}$, $\mathbb{E}=\mathbb{E}^{0}$  and write 
$Z\eid Y$ for two random variables $Z,Y$ with values in $\Rd$ to mean that they are equal 
in distribution or have the same law.  Throughout the paper, $\eta_{t}^{(\alpha/2)}(s)$ will stand for the transition density of the random variable $S_t$.

In the case $\alpha=2$, ${\bf X}$ is by definition a Brownian motion running at twice the usual speed and 
\begin{equation*}
p^{(2)}_t(x,y)=(4\pi t)^{-d/2}\myexp{-\frac{\abs{x-y}^2}{4t}}.
\end{equation*}
As for the cases $0<\alpha<2$, it is a standard fact  that the transition densities $\palp(x,y)$ can be written in terms of the $\alpha/2$-subordinator (see \cite[p.~522]{Acu} for further details). That is, 
\begin{align}\label{dbysub}
\palp(x,y)=\E{p^{(2)}_{S_t}(x,y)}=\int_0^{\infty} ds\,p^{(2)}_s(x,y) \,\eta_{t}^{(\alpha/2)}(s).
\end{align}
We remark at this point that identity \eqref{dbysub} will be useful
to extend results known for $\alpha=2$ in higher dimensions to the cases $1<\alpha<2$.

The transition densities $\palp(x,y)$ are known to have  an explicit expression only 
for  $\alpha=2$ and  $\alpha=1$.  In fact, for $\alpha=1$, the function
$p^{(1)}_t(x,y)$ is called the Cauchy (or Poisson in  analysis) heat kernel and  is given
by
\begin{align}\label{Cauchyk}
p^{(1)}_t(x,y)=\frac{k_d\,t}{\pthesis{t^2+\abs{x-y}^2}^{(d+1)/2}},
\end{align}  
where
\begin{align*}
k_d=\frac{\Gamma\pthesis{\frac{d+1}{2}}}{\pi^{\frac{d+1}{2}}}
\end{align*}
and the stochastic process ${\bf X}$ is called  Cauchy process.  
However, for the purposes of this paper, we only need to make use of  the following two facts about $\palp(x,y)$ for all $\alpr$. First, there exists $c_{\alpha,d}>0$ such that
\begin{align}\label{tcomp}
c_{\alpha,d}^{-1}\min\set{t^{-d/\alpha},\frac{t}{\abs{x-y}^{d+\alpha}}}\leq
\palp(x-y)\leq c_{\alpha,d}\min\set{t^{-d/\alpha},\frac{t}{\abs{x-y}^{d+\alpha}}},
\end{align}
for all $x,y\in \R^d$ and $t>0$ (see \cite{Chen1}).  Secondly, according to \cite[Theorem 2.1]{Blum}, we have
\begin{align}\label{tlim}
\lim_{\tgo}\frac{\palp(x-y)}{t}=\frac{A_{\alpha,d}}{\abs{x-y}^{d+\alpha}},
\end{align}
for all $x\neq y$, where 
\begin{equation}\label{Adef}
A_{\alpha,d}=\alpha \, 2^{\alpha-1}\, \pi^{-1-\frac{d}{2}}\,\sin\pthesis{\frac{\pi\alpha}{2}}\,\Gamma\pthesis{\frac{d+\alpha}{2}}\,\Gamma\pthesis{\frac{\alpha}{2}}.
\end{equation}

With all the necessary facts about the transition densities
$\palp(x,y)$ being properly recalled, we proceed to introduce the geometric objects where the stable processes will be studied. 
Let $\Omega\subset \R^d$ satisfy the following assumptions   according to the
dimension $d$ under consideration. If $d=1$, $\Omega$ will  be an open interval  $(a,b)$, $-\infty<a<b<\infty$ whose length $b-a$  will   be denoted  by $\ld$.
For  $d\geq 2$, the set $\Omega$ will be a uniformly $C^{1,1}$-regular  bounded domain  where $|\Omega|$ and  $\bd$ stand for the Lebesgue measure of $\Omega$ in $\Rd$
and its boundary, respectively. We recall that 
\begin{definition}
 $\Omega\subset \R^d$, $d\geq2$ with either finite or infinite Lebesgue measure and
non--empty boundary $\bd$ is said to be a uniformly $C^{1,1}$-regular set if there are constants $r, L > 0$ such that for every $ \sigma
\in \bd$, the set $ \bd \cap B_r(\sigma)$ is the graph of
a $C^{ 1,1}$--function $\Lambda$ with $||\nabla \Lambda||_{\infty}\leq  L$. Here and for the remainder of the paper, $B_r(\sigma)$ will represent the open ball   about $\sigma$ with radius $r$. 
\end{definition}

We point out   that according to Lemma 2.2  in \cite{Aik}, uniformly $C^{1,1}$--regular  bounded domains are  also $R$-smooth boundary domains. That is, for every $\sigma \in \bd$, there are two open balls $B_1$ and $B_2$ with radii $R$ such that $B_1 \subset \Omega$, $B_2 \subset \R^d \setminus\bar{\Omega}$ and $\partial B_1 \cap \partial B_ 2 = \sigma$. Henceforth,  for any $\Omega\subset \Rd$, we set
\begin{align}\label{Haus}
  \mathcal{H}^{d-1}(\bd)=\left\{
  \begin{array}{c c }
   \text{Hausdorff measure of the boundary of $\Omega$,} & \quad \text{if  $d\geq 2$, }\\ 
    \\
   \#\set{x\in \R: x\in \partial \Omega}, & \quad \text{if $d=1$.} \\ 
  \end{array} \right.
\end{align}
Of course, for $C^{1,1}$--domains as above,  this is just the surface area of the boundary of the domain.  

Let us consider for any Borel   measurable sets $\Omega, \Omega_0$ in $\R^d$, the following quantity
\begin{align}\label{hset}
\mathbb{H}^{(\alpha)}_{\Omega, \Omega_0}(t)=\int_{\Omega}dx\,\,\Prob^{x}\pthesis{X_t\in \Omega_0}=\int_{\Omega}dx\,\int_{\Omega_0}dy\,
\,\palp(x,y),
\end{align}
which turns out to be well  defined for example when either $\Omega$ or $\Omega_0$ has finite Lebesgue measure. 
When $\Omega=\Omega_0$, we simply denote $\myH_{\Omega,\Omega}^{(\alpha)}(t)$ by $\myH^{(\alpha)}_{\Omega}(t)$.

One of the goals of this paper is to study the small time behavior of the  function $\Halp$, which is equivalent
to analysing the behavior of $\myH_{\Omega,\Omega^c}^{(\alpha)}(t)$ as $\tgo$ since
\begin{align}\label{Hdef}
\Halp=\abs{\Omega}-\myH_{\Omega,\Omega^c}^{(\alpha)}(t).
\end{align}

We note that the function $u(t,x)=\int_{\Omega}\,dy\,\palp(x,y)$ is the unique weak solution to the initial value problem
\begin{align}\label{ivp}
\frac{\partial u}{
\partial t}&=-\F u(t,x), \,\,\,\, (t,x)\in (0,\infty)\times \R^d,\\
u(0,x)&=\mathbbm{1}_{\Omega}(x).\nonumber
\end{align}
Here, $\F$ denotes the fractional Laplacian and   the interested reader may consult \cite{Acu,  Acu3, Acu2,  BanSab, BanYil} for a detailed treatment and  applications to the theory of Schr\"odinger operators
and scattering theory.  In  other words, the initial value problem \eqref{ivp} exactly says that $\Halp$ represents the amount of heat in $\Omega$, 
if $\Omega$ is at initial temperature $1$ and if $\Omega^c$ is at initial temperature $0$.
In \cite{vanden3}, M. van den Berg  called $\myH^{(2)}_{\Omega}(t)$  the heat content of $\Omega$ in $\Rd$ and therefore  following the terminology introduced by M. van den Berg, we will also call $\Halp$  the heat content of $\Omega$ in $\Rd$.  We refer the interested reader to the papers \cite{van6, van7, van8} for recent results concerning bounds and asymptotic behaviors of the heat content corresponding to the Brownian motion over open sets, polygonal  domains and its extensions when dealing with compact manifolds.

We now proceed to interpret $\myH_{\Omega,\Omega^c}^{(\alpha)}(t)$ and discuss its connections with  a spectral function and the heat semi-group.  From definition \eqref{hset}, we observe  that $\Hset{\Omega}{\Omega^c}$  describes how fast in average the underlying stochastic process  ${\bf X}$, when started at some point inside of $\Omega$, escape from $\Omega$. When $\alpha =2$, as  previously mentioned, the process ${\bf X}$ is  the Brownian motion at twice speed whose  paths are continuous, whereas for $0<\alpha<2$, the paths of ${\bf X}$ are right continuous with left limits. Thus, $\Hset{\Omega}{\Omega^c}$, by definition, is related to  the jumps or the fluctuation of   the paths up to time $t$ of the corresponding process  under consideration. 

The interest in studying $\Hset{\Omega}{\Omega^c}$ derives from the results  known
about $\myH_{\Omega, \Omega^c}^{(2)}(t)$ in higher dimensions which  we proceed to mention. We consider  the  heat semi-group acting on $L^{2}(\R^d)$ associated with the
process ${\bf X}$. Namely,
\begin{equation}
T_t^{(\alpha)}(f)(x)=\int_{\R^d}dy\,\,f(y)\,\,p^{(\alpha)}_t(x-y)=\E{f(x-X_t)}.
\end{equation}
Therefore, it follows from \eqref{hset} that 
\begin{equation}
\myH_{\Omega,\Omega^c}^{(\alpha)}(t)=\left<T_t^{(\alpha)}(\mathbbm{1}_{\Omega}),\mathbbm{1}_{\Omega^c}\right>,
\end{equation}
where $\langle\cdot,\cdot\rangle$ denotes the standard inner product in $L^{2}(\R^d)$.
In \cite{Mir} and \cite{Pre}, M. Miranda, D. Pallara,  F. Paronetto and M. Preunkert have investigated for the Brownian motion case $\alpha=2$  the connections between
 $\myH^{(2)}_{\Omega,\Omega^c}(t)$, functions of bounded variation and   the isoperimetric inequality by means of analytic tools when $d\geq 2$ for not only uniformly $C^{1,1}$-regular bounded domains  but also bounded  Cacciopoli sets.  In fact, it is shown in \cite[Proposition 8]{Pre}
  that
\begin{align}\label{isop}
\frac{\myH^{(2)}_{\Omega,\Omega^{c}}(t)}{\sqrt{t}}\leq
\frac{1}{\sqrt{\pi}}\mathcal{H}^{d-1}(\bd)
\end{align}
for all $t>0$, while in  \cite[Theorem 2.4]{Mir}  is proved  that 
\begin{equation}\label{hc2}
\lim\limits_{\tgo}\frac{\myH^{(2)}_{\Omega,\Omega^{c}}(t)}{\sqrt{t}}=\frac{1}{\sqrt{\pi}}\mathcal{H}^{d-1}(\bd).
\end{equation}
 Consequently, the preceding limit and \eqref{Hdef} yield the following asymptotic expansion for such domains,
\begin{equation*}
\myH^{(2)}_{\Omega}(t)=|\Omega|-\frac{1}{\sqrt{\pi}}\mathcal{H}^{d-1}(\bd)\,t^{\frac{1}{2}}+o(t^{\frac{1}{2}}), \,\,\, \tgo.
\end{equation*}
The main observation here is that we are able to recover a geometry feature of the set $\Omega$ in addition to its volume 
from the small asymptotic expansion of $\myH^{(2)}_{\Omega}(t)$, namely, the surface area of its boundary $\bd$. 

The  main purpose in this paper is to investigate the asymptotic behavior of $\myH^{(
\alpha)}_{\Omega}(t)$, $0<\alpha<2$ as $\tgo$ and try to recover information on the
geometry of $\Omega$. As we will see later, we are able to  recover the surface area of the boundary if $1\leq \alpha<2$ and
 the fractional $\alpha$--perimeter  when $0<\alpha<1$, (see \eqref{alpper} below for definition of the the fractional perimeter) from the small time behavior of the heat content.
 
 For $d=1$, our main result  is the
 following.

\begin{thm} \label{mth} Let $\Omega=(a,b)$, $-\infty <a<b<\infty$ and $\ld=b-a$.
\item[$(i)$] For $1<\alpha\leq 2$ and  all $t>0$, we have
\begin{align*}
\Hset{\Omega}{\Omega^c}=\frac{2}{\pi}\,\Gamma\pthesis{1-\inalp}t^{\inalp} + R_{\alpha}(t),
\end{align*}
with 
\begin{align*}
\abs{R_{\alpha}(t)}\leq C\pthesis{\,t\,\mathbbm{1}_{(1,2)}(\alpha)+t^{3/2}
\,\mathbbm{1}_{\set{2}}(\alpha)}.
\end{align*}
\item[$(ii)$] For $\alpha=1$ and all $t>0$, the following  equality holds.
\begin{align*}
\myH_{\Omega, \Omega^c}^{(1)}(t)=\frac{2}{\pi}\,t\ln\pthesis{\frac{1}{t}}
+\frac{2}{\pi}\pthesis{\ld \,\arctan\pthesis{\frac{t}{\ld}}+
\frac{1}{2}\,\,t\,\ln\pthesis{t^{2}+\ld^{2}}}, 
\end{align*}
\item[$(iii)$] Let $0<\alpha<1$ and $0<t<\min\set{\ld^{\alpha}, e^{-1}}$. We obtain the subsequent expansions according to the following sub-cases.
\begin{enumerate}
\item[$(iv)$] If $1/\alpha \notin \mathbb{N}$, then there is  a constant $C_{\alpha}$ independent of $\Omega$ such that
\begin{align*}
\Hset{\Omega}{\Omega^c}=\frac{2}{\pi}\mysum{n}{1}{\pint{\inalp}}(-1)^{n-1}\frac{
\Gamma(n\alpha)}{(1-n\alpha)n!}\ld^{1-n\alpha}\sin\pthesis{\frac{\pi n\alpha}{2}}t^{n}+\,\,C_{\alpha}\,t^{\inalp}
+R_{\alpha}(t),
\end{align*}  
with
$\abs{ R_{\alpha}(t)}\leq C \,\,t^{\left[\inalp \right]+1}$.
\\
\item[$(v)$] If $\alpha=1/N$, for some $N \in \mathbb{N}$, then
there is a constant $C_{N}(\Omega)$ such that
\begin{align*}
\myH^{(1/N)}_{\Omega,\Omega^c}(t)&= \frac{2}{\pi}\mysum{n}{1}{N-1}(-1)^{n-1}
\frac{
\Gamma(n/N)}{(1-n/N)n!}\ld^{1-n/N}\sin\pthesis{\frac{\pi n}{2N}}t^{n}
\\&+\,(-1)^{N-1}\,\frac{2}{\pi (N-1)!}\,t^{N}\ln\pthesis{\frac{1}{t}}+ C_{N}(\Omega)\,t^{N}
+R_{1/N}(t),
\end{align*}
with $\abs{R_{1/N}(t)}\leq C\,\, t^{N+1}$.
\end{enumerate}
In all the above statements, $C>0$ depends only on $\alpha$ and $\Omega$.
\end{thm}
It is worth mentioning that the case $\alpha=2$ has been included in the above theorem because  to the best of our knowledge, $\myH^{(2)}_{\Omega}(t)$ has only been studied 
in detail when $\Omega$ is a domain of $\Rd$ with $d\geq 2$ omitting the elemental case when $\Omega$ is an interval with finite length.

We also notice that Theorem \ref{mth} ensures on one hand the existence of a non-zero function $h_{\alpha}(t)$ such
that
\begin{align}\label{lim1dim}
\lim\limits_{\tgo}\frac{\Hset{\Omega}{\Omega^c}-h_{\alpha}(t)}{t^{\inalp}}
\end{align}
exists for all $0<\alpha\leq 1$. On the other hand, for $1<\alpha\leq2$ the 
above limit also exists with $h_{\alpha}(t)=0$. 

 The upcoming Theorem \ref{2thm} will show that the preceding limit \eqref{lim1dim} also exists in higher dimensions for $1<\alpha<2$ whereas for $0<\alpha\leq1$ we are only able to obtain a weaker version of the statements $(ii)$
and $(iii)$ provided in Theorem \ref{mth}. For $\alpha=1$, it is worth noting that theorems \ref{mth} and \ref{2thm} indicate that $h_{1}(t)$ should be  equal to
$$\frac{1}{\pi}\,\,
\mathcal{H}^{d-1}(\partial \Omega)\,\,t\,\,\ln\pthesis{\frac{1}{t}}.$$
 The main difficulty here would consist in identifying the limit \eqref{lim1dim}.

We now continue to elaborate further in the observation  previously made. We point out that the factor $2$ which appears in the first term of each expansion in Theorem
\ref{mth} comes from the boundary points of the interval $(a,b)$ and 
by definition  \eqref{Haus}, we have $\mathcal{H}^{0}(\partial (a,b))=2$. With simple observation, we notice that part $(i)$ can be restated  as
\begin{align*}
\lim\limits_{\tgo}\frac{\Hset{\Omega}{\Omega^c}}{t^{\inalp}}=\frac{1}{\pi}\Gamma\pthesis{1-\inalp}\mathcal{H}^{0}(\bd).
\end{align*}
For $\alpha=2$, the above limit is the one-dimensional  analogue of $\eqref{hc2}$ with
the same constant which is not unusual since when dealing with a $d$-dimensional Brownian motion most of the computations reduce to the one  dimensional setting due to the independence of the components. However, for $0<\alpha<2$ the components are no longer independent and  an approach involving estimates of the heat kernels is required.

Because of the last considerations, we are led to conjecture that in higher dimensions we should expect to recover, with the first term of each expansion,  the Hausdorff measure of the boundary. Our Theorem \ref{2thm} asserts that the conjecture
is correct when $1\leq \alpha<2$ with the same constants as in  parts $(i)$ and $(ii)$ of Theorem \ref{mth}.   For $0<\alpha<1$,  the fractional $\alpha$-perimeter $\mathcal{P}_{\alpha}\pthesis{\Omega}$, defined to be
\begin{align}\label{alpper}
\myP_{\alpha}\pthesis{\Omega}=\int_{\Omega}\int_{\Omega^c}\frac{dx\,\,dy}{|x-y|^{d+\alpha}},
\end{align} 
is recovered.
The aforementioned quantity turns out to be linked with celebrated Hardy inequalities.  We refer the  reader to the papers of Z. Q. Chen, R. Song \cite{Song} and R. L. Frank, R. Seiringer \cite{Frank} for  further results involving this quantity. In fact, it is shown in \cite{Frank} that there  exists $\lambda_{d,\alpha}>0$ such that
\begin{align*}
|\Omega|^{(d-\alpha)/d}\leq \lambda_{d,\alpha}\,\myP_{\alpha}\pthesis{\Omega}
\end{align*}
with equality if and only if $\Omega$ is a ball.
It is also proved in \cite{FracP} that
\begin{align*}
\lim_{\alpha\downarrow 0}\alpha \myP_{\alpha}(\Omega)=d\,\,|B_1(0)|\,\,|\Omega|,\nonumber \\
\lim_{\alpha\uparrow 1}(1-\alpha)\myP_{\alpha}(\Omega)=K_{d}\,\,\mathcal{H}^{d-1}(\partial \Omega),
\end{align*}  
for some $K_d>0.$

It is interesting to notice that the last limit intuitively gives an insight  that the surface area of the boundary should be recovered when considering the small time behavior of the function $\myH_{\Omega, \Omega^c}^{(1)}(t)$(Cauchy process,  $\alpha=1$) which is exactly what our next result shows. 

\begin{thm}\label{2thm}
Assume   $\Omega\subset \Rd$, $d\geq 2$ is a uniformly $C^{1,1}$-regular  bounded domain.
\item[$(i)$] For $1<\alpha<2$, we have
\begin{align}\label{iso2}
\myH_{\Omega,\Omega^c}^{(\alpha)}(t)\leq 
\frac{1}{\pi}\,\Gamma\pthesis{1-\inalp}\,\mathcal{H}^{d-1}(\partial \Omega)\, t^{\inalp}
\end{align}
for all $t>0$. Moreover,
\begin{align}\label{Isoplim}
\lim\limits_{\tgo}\frac{\myH_{\Omega,\Omega^c}^{(\alpha)}(t)}{t^{\inalp} }= 
\frac{1}{\pi}\,\Gamma\pthesis{1-\inalp}\,\mathcal{H}^{d-1}(\partial \Omega).
\end{align}

\item[$(ii)$] For $\alpha=1$, 
\begin{align*}
\lim_{\tgo} \frac{ \myH^{(1)}_{\Omega, \Omega^c}(t)}{t\ln\pthesis{\frac{1}{t}}}=\frac{1}{\pi}\,\,
\mathcal{H}^{d-1}(\partial \Omega).
\end{align*}

\item[$(iii)$] For $0<\alpha<1$, 
\begin{align*}
\lim_{\tgo} \frac{ \myH^{(\alpha)}_{\Omega, \Omega^c}(t)}{t}=
A_{\alpha,d}\,\,\myP_{\alpha}(\Omega),
\end{align*}
with $A_{\alpha,d}$ and $\myP_{\alpha}(\Omega)$ as defined in \eqref{Adef} and \eqref{alpper}, respectively.
\end{thm}

The proof of $(i)$ is a consequence of  the Lebesgue Dominated Convergence Theorem 
and subordination techniques.  Part $(iii)$ is obtained by combining once again the Lebesgue Dominated Convergence Theorem with \eqref{tlim}. The case $\alpha=1$ requires a more elaborate  approach.

We next establish some  connections between $\Hset{\Omega}{\Omega^c}$ 
and  the spectral heat content of $\Omega$ which  has only  been widely studied for the Brownian motion case.  Denote by 
 \begin{equation*}\label{exittime}
\tau_{\Omega}^{(\alpha)}=\inf\set{s\geq 0:X_s\in \Omega^c},
\end{equation*}   
the first exit time from $\Omega$. The \textit{spectral heat content} of $\Omega$, denoted by $Q_{\Omega}^{(\alpha)}(t)$,
is defined as
\begin{align}\label{shc}
{Q}_{\Omega}^{(\alpha)}(t)=\int_{\Omega}dx\,\int_{\Omega}dy\,\,p^{\Omega,\alpha}_t(x,y),
\end{align}
where $p^{\Omega,\alpha}_t(x,y)$ is the transition density for the stable process killed upon exiting $\Omega$. More precisely, this is the heat kernel for the Dirichlet fractional Laplacian.  An explicit expression is given by 
\begin{equation}\label{tran.den.dom}
p^{\Omega,\alpha}_t(x,y)=\palp(x,y)\,\,\mathbb{P}\left(\tau_{\Omega}^{(\alpha)}>t \,\,|\,\, X_0=x,\,\,X_t=y\right).
\end{equation}
The name {\it spectral heat content}  given to $Q_{\Omega}^{(\alpha)}(t)$ comes from the fact that $p^{\Omega,\alpha}_t(x,y)$ can be written in terms of the eigenvalues and eigenfunctions of the domain 
 $\Omega$. That is, when  $|\Omega|<\infty$, it is known (see \cite{Davies} for details) that there exists an orthonormal basis of eigenfunctions $\{\phi_n \}_{n\in \mathbb{N}}$ for $L^2(\Omega)$  with  corresponding eigenvalues $\{\lambda_n \}_{n \in \mathbb{N}}$
satisfying $0 < \lambda_1 < \lambda_2 \leq \lambda_3 \leq . . .$
and  $\lambda_n \rightarrow \infty$ as $n \rightarrow \infty$ such that
\begin{equation}\label{spect.dec.palp}
p^{\Omega,\alpha}_t(x,y)=\mysum{n}{1}{\infty}e^{-t\lambda_n}\,\phi_{n}(x)\,\phi_n(y).
\end{equation}

Notice that due to \eqref{shc} and the last equality, we obtain an expression for 
$Q_{\Omega}^{(\alpha)}(t)$ involving both the spectrum $\set{ \lambda_n}_{ n\in \mathbb{N}}$ and  eigenfunctions $\set{\phi_{n}}_{n\in \mathbb{N}}$. Namely,
$$Q_{\Omega}^{(\alpha)}(t)=\mysum{n}{1}{\infty}e^{-t\lambda_{n}}\pthesis{\int_{\Omega}dx\,\phi_n(x)}^2.$$

We remark for the sake of completeness that by mimicking the proof provided in \cite[Proposition 1.4]{van2}, we have
$$Q_{\Omega}^{(\alpha)}(t)=e^{-\lambda_1t}\pthesis{ ||\phi_1||_1^2+\mathcal{O}(t^{-d/\alpha})}, \,\,t\uparrow\infty.$$ 
 Henceforth, we will only be concerned about   the behavior of $Q_{\Omega}^{(\alpha)}(t)$ as $\tgo$.
 
 The study of the small time behavior of the  spectral heat content 
$Q_{\Omega}^{(\alpha)}(t)$ arises from  the results  associated with the asymptotic expansion of the heat trace for smooth domains. The heat trace
of a bounded domain  $\Omega$ is defined to be
\begin{align*}
\mathcal{Z}_{\Omega}^{\,(\alpha)}(t)=\frac{1}{p_t^{(\alpha)}(0)}\int_{\Omega}\,dx\,p_t^{\Omega, \alpha}(x,x)=\frac{1}{p_t^{(\alpha)}(0)} \mysum{n}{1}{\infty}e^{-\lambda_n t},
\end{align*}
where the second equality is obtained by means of \eqref{spect.dec.palp}. In \cite{BanKul}, R. Ba\~nuelos and T. Kulczycki provide the following second order expansion of the heat trace  for $R$-smooth boundary domains which holds  for every $0<\alpha\leq2$ (the case $\alpha=2$ was proved  in \cite{van5} by M. van der Berg).  
\begin{align}
\mathcal{Z}_{\Omega}^{\,(\alpha)}(t)=
\abs{\Omega} -C_{d,\alpha}\,\mathcal{H}^{d-1}(\bd)\,t^{\inalp}+\mathcal{O}(t^{\frac{2}{\alpha}}),
\end{align}
as $\tgo$, where $C_{d,\alpha}>0$ admits a probabilistic representation in terms of the exit time from  the upper half--plane of the underlying $\alpha$--stable process. This result was extended by R.  Ba\~nuelos, T. Kulczycki and  B. Siudeja to domains with Lipschitz boundaries in \cite{BanKulSiu}.
It is interesting to note  that the above expansion for $0<\alpha<2$ was motivated by scaling and keeping in mind the behavior of the heat trace for the Brownian motion.  Based on this, it is natural to predict the second order expansion  of $Q_{\Omega}^{(\alpha)}(t)$ by considering as a model the spectral heat content of the Brownian motion $Q_{\Omega}^{(2)}(t)$. To our surprise (as we shall see below)  $Q_{\Omega}^{(2)}(t)$ only models the behavior of  $Q_{\Omega}^{(\alpha)}(t)$ for the cases $1<\alpha<2$.

The small time asymptotic behavior of 
$Q_{\Omega}^{(\alpha)}(t)$ is known so far only for $\alpha=2$. In fact, the following result was proved by M. van den Berg and J. F. Le Gall in  \cite{van2} for smooth domains $\Omega\subset \R^d$, $d\geq 2$. 
\begin{equation}\label{VandenberHeatContBM}
Q_{\Omega}^{(2)}(t)=|\Omega|-\frac{2}{\sqrt{\pi}}\mathcal{H}^{d-1}(\bd)\,t^{1/2}+\left(\frac{d-1}{2}\int_{\bd}\mathcal{M}(\sigma)d\sigma\right) t +\mathcal{O}(t^{3/2}),  
\end{equation}
as $t\downarrow 0$.  Here, $\mathcal{M}(\sigma)$ denotes the mean curvature at the point $\sigma\in \bd$.  
For more on the heat content asymptotics and its connections to the eigenvalues (spectrum)  of the Laplacian in the domain $\Omega$, we direct the reader to Gilkey's monograph \cite{Gil1} and to M. van den Berg, E. B. Dryden and T. Kappeler \cite{vandenDryKap} and the many references to the literature contained therein.  We also refer the reader to \cite{van4} for matters related to the spectral heat content  of the Brownian motion for regions with a fractal boundary.

For  $\Omega$ a uniformly $C^{1,1}$--regular bounded domain is  known, according to
\cite[Corollary 1]{Bog2}  that there exists $c>0$ such that
\begin{align*}
c^{-1}\min\set{1, \frac{\rho^{\alpha/2}_{\Omega}(x)}{\sqrt{t}}}\leq \int_{\Omega}dy\,\,p^{\Omega,\alpha}_t(x,y)\leq
c \min\set{1, \frac{\rho^{\alpha/2}_{\Omega}(x)}{\sqrt{t}}},
\end{align*}
for all $x\in \Omega$ and $0<t \leq 1$. Here, $\md$ represents the distance from $x$ to 
the boundary of $\Omega$.  Therefore, for bounded domains $\Omega$ with smooth boundary $\bd$, it is possible to  prove by using the techniques developed in \cite{van2} that
\begin{align}\label{conjeture}
\int_{\Omega} dx \,\min\set{1, \frac{\rho^{\alpha/2}_{\Omega}(x)}{\sqrt{t}}}=\abs{\Omega}-
C_{\alpha}\mathcal{H}^{d-1}(\bd)\,\, t^{\inalp}+
\mathcal{O}(t^{\frac{2}{\alpha}}),
\end{align}
as $\tgo$ for some $C_{\alpha}>0$. Hence,  based on the preceding expansion and the 
small time expansion  \eqref{VandenberHeatContBM} for the Brownian motion, we are led to conjecture that
a  similar asymptotic expansion to the right hand side of \eqref{conjeture} should
also hold for  $Q_{\Omega}^{(\alpha)}(t)$. However, the following theorem  asserts that such a conjecture may only hold for  $1<\alpha<2$. 
\begin{thm}\label{mcor}
 Assume   $\Omega\subset \Rd$, $d\geq 2$ is a uniformly $C^{1,1}$-regular  bounded  domain.
\begin{enumerate}
\item[$(i)$] Let $1<\alpha<2$. Then,  we have
\begin{align*}
\frac{1}{\pi}\Gamma\pthesis{1-\inalp}
\mathcal{H}^{d-1}(\bd)&\leq
\varliminf\limits_{\tgo}\frac{|\Omega|-Q_{\Omega}^{(\alpha)}(t)}{t^{\inalp}}
\\
&\leq \varlimsup\limits_{\tgo}\frac{|\Omega|-Q_{\Omega}^{(\alpha)}(t)}{t^{\inalp}}
\leq 2^{(3d+1)/2}\,\Gamma\pthesis{1-\inalp}\mathcal{H}^{d-1}(\bd).
\end{align*}
\item[$(ii)$] For $\alpha=1$, we obtain
\begin{align*}
\frac{1}{\pi} \,\,\mathcal{H}^{d-1}(\bd)&\leq 
\varliminf\limits_{\tgo} 
\frac{|\Omega|-Q_{\Omega}^{(1)}(t)}{t\,\ln\pthesis{\frac{1}{t}}}\\
&\leq \varlimsup\limits_{\tgo} 
\frac{|\Omega|-Q_{\Omega}^{(1)}(t)}{t\,\ln\pthesis{\frac{1}{t}}}\leq\,
2^{(3d+1)/2}\,\,\mathcal{H}^{d-1}(\bd).
\end{align*}
\item[$(iii)$] For $0<\alpha<1$, there exists a positive constant $C_{d,\alpha}$ such that
\begin{align*}
A_{d,\alpha}\,\myP_{\alpha}(\Omega)&\leq 
\varliminf\limits_{\tgo} 
\frac{|\Omega|-Q_{\Omega}^{(\alpha)}(t)}{t} \\
&\leq \varlimsup\limits_{\tgo} 
\frac{|\Omega|-Q_{\Omega}^{(\alpha)}(t)}{t}\leq \,C_{d,\alpha}
\,\int_{\Omega}\,dx\, \rho^{-\alpha}_{\Omega}(x),
\end{align*} 
where $\rho_{\Omega}(x)=\inf\set{\abs{\sigma-x}:\sigma\in \bd}$. Moreover, if $\Omega$ satisfies a uniform  exterior volume condition, the quantity $\int_{\Omega}\,dx\, \rho^{-\alpha}_{\Omega}(x)$ can be replaced up to some positive constant by $\myP_{\alpha}(\Omega)$.
 
Here $A_{\alpha,d}$ and  $\myP_{\alpha}(\Omega)$ as defined in \eqref{Adef} and \eqref{alpper} respectively. 
\end{enumerate}
\end{thm}

The lower bounds obtained in the foregoing theorem  are an easy consequence of applying Theorem \ref{2thm} together with
the following inequality which arises from equality \eqref{tran.den.dom}(where we have appealed to the fact that the conditional probability is bounded by $1$) and it relates the heat content $\Halp$ with  the spectral heat content $Q_{\Omega}^{(\alpha)}(t)$ as follows.
\begin{align*} 
Q_{\Omega}^{(\alpha)}(t)\leq \Halp=|\Omega|-\Hset{\Omega}{\Omega^c},
\end{align*}
for all $t>0$. On the other hand, the upper bounds require  a more delicate treatment where the $\alpha/2$-subordinator  plays a relevant role.
Based on the preceding estimates and the last theorem, we state  the following conjecture
about the small time behavior for the spectral heat content of $\Omega$.
\\

{\it {\bf Conjecture}}
\begin{enumerate}
\item[$(i)$] For $1<\alpha<2$, there exists $C_{d,\alpha}>0$ such that
\begin{align*}
Q_{\Omega}^{(\alpha)}(t)&=\abs{\Omega}-C_{d,\alpha}\,
\mathcal{H}^{d-1}(\bd) \, t^{\inalp} +\mathcal{O}(t),\,\, \tgo.
\end{align*}

\item[$(ii)$] For $\alpha=1$, there exists $C_{d}>0$ such that
\begin{align*}
\hspace{7mm}Q_{\Omega}^{(\alpha)}(t)&=\abs{\Omega}-C_{d}\,
\mathcal{H}^{d-1}(\bd) \, t\,\ln\pthesis{\frac{1}{t}} +\mathcal{O}(t), \,\, \tgo.
\end{align*}

\item[$(iii)$] For $0<\alpha<1$, there exists $C_{d,\alpha}>0$ such that
\begin{align*}
Q_{\Omega}^{(\alpha)}(t)&=\abs{\Omega}-C_{d,\alpha}\myP_{\alpha}(\Omega) \, t +o(t),\,\, \tgo.
\end{align*}
\end{enumerate}

The rest of the paper is organized as follows. In \S \ref{sec:T1}, we provide the proof of Theorem \ref{mth}. In \S \ref{sec:T2}, we  show part
$(i)$ of Theorem \ref{2thm} by means of subordination techniques.
In \S \ref{sec:Cauchy}, we develop some machinery  for 
the Cauchy heat kernel and half-planes and finish the proof of Theorem \ref{2thm}. Finally, in \S \ref{sec: ub}, the proof
of Theorem \ref{mcor} is given where the $\alpha/2$-subordinator
${\bf S}$ plays a crucial role.


\section{proof of theorem \ref{mth}} \label{sec:T1}
We will begin this section by presenting some fundamental properties about the  $\alpha/2$--subordinator ${\bf S}=\set{S_t}_{t\geq0}$.  
\begin{prop}\label{factssubor}
\hspace*{20mm}
\begin{enumerate} 
\item[$(i)$] For all $\lambda,t >0$,
\begin{equation*}\label{Lap.sub}
 \E{\exp\pthesis{-\lambda S_{t}}}=\exp\pthesis{ -t\lambda^{\alpha/2}}.
\end{equation*}
\item[$(ii)$] For all $-\infty <\beta< \frac{\alpha}{2}$, 
\begin{equation}\label{gen.expt.S1}
\E{S_{1}^{\beta}}=\int_{0}^{\infty}\,ds\,s^{\beta}\,\eta^{(\alpha/2)}_1(s)=\frac{\Gamma(1-\frac{2\beta}{\alpha})}{\Gamma(1-\beta)}.
\end{equation}
\item[$(iii)$]
Let $\kappa>0$. Then, there exists $C_{\alpha}>0$ such that
\begin{equation}\label{basicupperest}
\E{\exp\pthesis{-\frac{\kappa^2}{ S_1}}}\leq C_{\alpha}\,\,\kappa^{-\alpha}.
\end{equation}
\end{enumerate}
\end{prop}

\begin{proof}
$(i)$ and $(ii)$ are standard facts so that we refer the reader to \cite{Acu} for the proofs.

Regarding $(iii)$, it is known (see \cite[p~97]{Bog}) that $\eta^{(\alpha/2)}_1(s)\leq C_0(\alpha)\min\set{1,s^{-1-\frac{\alpha}{2}}}$ for some $C_0(\alpha)>0$.
Hence, after a suitable change of variables, we arrive at
\begin{align*}
\E{\exp\pthesis{-\frac{\kappa^2}{ S_1}}}&\leq C_0(\alpha)\int_{0}^{\infty}\exp\pthesis{ -\kappa^2/s}
\min\set{1,s^{-1-\frac{\alpha}{2}}}ds \\
&\leq C_0(\alpha)\kappa^2 \int_{0}^{\infty}\exp\pthesis{-w}\set{\frac{w}{\kappa^2}}^{1+\frac{\alpha}{2}}\frac{dw}{w^2}=\frac{C_0(\alpha)\Gamma(\alpha/2)}{\kappa^{\alpha}}.
\end{align*}
Thus, the proof is complete by  taking $C_{\alpha}=C_0(\alpha)\Gamma(\alpha/2)$.
\end{proof}

In what follows,  we shall assume that $\Omega=(a,b)$, $-\infty<a<b<\infty$ with length $b-a=\ld$. We start  by expressing $\Hset{\Omega}{\Omega^c}$ in a more convenient form. For this purpose, we require the following two fundamental identities concerning the process ${\bf X}$ which can be easily deduced from the characteristic function \eqref{CharF.Proc}. 
\begin{align*}
\Prob^{x}\pthesis{  X_t \in A}&=
\Prob\pthesis{x-t^{\inalp}X_1\in A},\\ \nonumber
\Prob\pthesis{  X_t\in A}&=\Prob\pthesis{  -X_t\in A},
\end{align*}
for all $t>0$, $x \in \R^d$ and $A$ being a Borel measurable set in 
$\R^d$. In  particular, when $d=1$, we obtain 
\begin{align*}
\Prob^{x}\left(X_t\leq a\right)&=\Prob\left((x-a)t^{-\inalp} \leq X_1 \right),\\
\Prob^{x}\left(b\leq X_t\right)&=
\Prob\left( (b-x)t^{-\inalp}\leq X_1\right),
\end{align*} 
for all $x,a, b \in \R$ and $t>0$.
The last identities  in turn imply  that
\begin{align*}
\Hset{\Omega}{\Omega^c}
&=\int_{a}^{b}dx\,\left[\Prob^{x}\left(X_t\leq a\right)+\Prob^{x}\left(b\leq X_t\right)\right] \\
&=\int_{a}^{b}dx\,\,\Prob\left((x-a)t^{-\inalp}\leq X_1 \right)+
\int_{a}^{b}dx\,\,\Prob\left((b-x)t^{-\inalp}\leq X_1 \right).
\end{align*}
Next, a simple change of variables yields
\begin{align*}
\Hset{\Omega}{\Omega^c}=2\,t^{\inalp}\,\int_{0}^{\ld t^{-\inalp}}dw\,\, \Prob\left(w\leq X_1 \right),
\end{align*}
which shows  that $\Hset{\Omega}{\Omega^c}$ is  related to the tail behavior of the process ${\bf X}$. 

We set
\begin{align}\label{eldef}
\ele=\int_{0}^{\ld t^{-\inalp}}dw\,\,\Prob\left(w\leq X_1 \right).
\end{align}

{\bf Proof of Theorem \ref{mth}:} Since the tail behavior of the Brownian motion and  stable processes have an exponential and an algebraic decay at infinity, respectively, we need to treat the cases $1<\alpha\leq2$, $\alpha=1$  and $0<\alpha<1$ separately.
\bigskip

{\bf Case $1<\alpha\leq 2$}:
We rewrite $\ele$ as a double integral as follows.
\begin{align*}
\ele=\int_{0}^{\ld t^{-\inalp}}dw\int_{w}^{\infty}dz\,
p_1^{(\alpha)}(z).
\end{align*}

Thus, by interchanging the order of integration, we arrive at
\begin{align*}
\ele&=\int_{0}^{\ld t^{-\inalp}}dz\, p_1^{(\alpha)}(z)\int_{0}^{z}dw
+ \int_{\ld t^{-\inalp}}^{\infty}dz\, p_1^{(\alpha)}(z)\int_{0}^{\ld t^{-\inalp}}dw \\
&=\int_{0}^{\ld t^{-\inalp}}dz\,z\, p_1^{(\alpha)}(z)
+ \ld t^{-\inalp}\int_{\ld t^{-\inalp}}^{\infty}\,
dz\,p_1^{(\alpha)}(z).
\end{align*}

In probabilistic terms, we have shown that
\begin{align*}
\ele&=\E{X_1,0\leq X_1\leq \ld t^{-\inalp} } + 
 \ld t^{-\inalp}\,\Prob\pthesis{ \ld t^{-\inalp}< X_1}\\
 &=\E{X_1, 0\leq X_1} -\E{X_1, \ld t^{-\inalp}< X_1 } +
 \ld t^{-\inalp}\,\Prob\pthesis{\ld t^{-\inalp}< X_1}.
\end{align*}

Let us denote 
\begin{align}\label{rm}
j_{\alpha}(t)=\E{X_1, \ld t^{-\inalp}< X_1 }
\end{align}
and observe that
\begin{align*}
j_{\alpha}(t)\geq \ld t^{-\inalp}\,\Prob\pthesis{\ld t^{-\inalp}< X_1}.
\end{align*}
Thus, the remainder function $R_{\alpha}(t)$ defined as
$$R_{\alpha}(t)=2t^{\inalp}\pthesis{-\E{X_1, \ld t^{-\inalp}< X_1 } +
 \ld t^{-\inalp}\,\Prob\pthesis{\ld t^{-\inalp}\leq X_1}},$$
satisfies $\abs{R_{\alpha}(t)}\leq 4\,t^{\inalp}\,j_{\alpha}(t).$ Therefore, to finish the proof of  part $(i)$
of Theorem \ref{mth}, it suffices to obtain  upper bounds for the  function $j_{\alpha}(t)$ according to the cases $\alpha=2$ and $1<\alpha<2$.

\bigskip
{\bf Case $\alpha=2$}: It is clear from \eqref{rm} that  
\begin{align*}
j_{2}(t)&=(4\pi)^{-1/2}
\int_{ \ld t^{-1/2}}^{\infty}dz\,z\exp\left(-\frac{z^2}{4}\right)=\pi^{-1/2}\exp\left(-\frac{\ld^2}{4t}\right).
\end{align*}

Next, by applying the elementary inequality 
\begin{align*}
\exp(-x)\leq x^{-1},\,\, x>0,
\end{align*}
we conclude that $j_{2}(t)\leq 4\pi^{-1/2}\ld^{-2}t.$  Hence, we have shown that  
\begin{align*}
\myH_{\Omega,\Omega^c}^{(2)}(t)=2\,\E{X_1, 0\leq X_1}\,t^{1/2} + R_{2}(t),
\end{align*}
with $\abs{R_{2}(t)}\leq C\,t^{3/2}$ for all $t>0$.
\\

{\bf Case $1<\alpha<2$}: 
We observe  because of \eqref{tcomp} that for all $z\in\R\setminus\set{0}$ we have
\begin{align*}
p_1^{(\alpha)}(z)\leq c_{\alpha,1} \abs{z}^{-1-\alpha}
\end{align*} 
so that 
\begin{align*}
j_{\alpha}(t)\leq c_{\alpha,1}\int_{\ld t^{-\inalp}}^{\infty}dz\,
z^{-1-\alpha}\,z=c_{\alpha,1}(\alpha-1)^{-1}\ld^{1-\alpha}
t^{1-\inalp}.
\end{align*}
Thus, we arrive at
\begin{align*}
\Hset{\Omega}{\Omega^c}=2\,\E{X_1, 0\leq X_1}\,t^{\inalp}+ R_{\alpha}(t), 
\end{align*}
with $\abs{R_{\alpha}(t)}\leq C\,t$ for all $t>0$.
\begin{rmk}
By combining  \eqref{dbysub} and Fubini's Theorem, we  obtain for all $1<\alpha\leq 2$ that 
\begin{align*}
\E{X_{1},0\leq X_1}&=\int_{0}^{\infty}
dz\,z\,p_{1}^{(\alpha)}(z)=\int_{0}^{\infty}
dz\,z\,\E{p_{S_1}^{(2)}(z)}\\
&=\E{ \int_{0}^{\infty}dz\,z\,p_{S_1}^{(2)}(z)}
=\frac{1}{\sqrt{\pi}}\,\,\E{S_{1}^{1/2}}=\frac{1}{\pi}\,\,\Gamma\left(1-\frac{1}{\alpha}\right),
\end{align*}
where in the last equality we have appealed to formula \eqref{gen.expt.S1}.
\end{rmk}

We proceed to deal with Cauchy processes.
\bigskip

{\bf Case $\alpha=1$}: We begin by recalling some elementary calculus identities. 
\begin{align}
\arctan(w)+\arctan\pthesis{\frac{1}{w}}&=\frac{\pi}{2}, \label{ei1}\\
\label{ei2}
\int dw\,\arctan(w)=w\arctan(w)&-\frac{1}{2}\ln\left(1+w^2\right)+C.
\end{align}
 
By appealing to the above identities, the explicit expression of the  Cauchy heat kernel
\eqref{Cauchyk}  and \eqref{eldef}, we have
\begin{align*}
\ell_{1}(t)&=\int_{0}^{\ld t^{-1}}dw\,\int_{w}^{\infty}\frac{dz}{\pi(1+z^2)}
\\
&=\frac{1}{\pi}\pthesis{\frac{\pi}{2}\ld t^{-1}-
\int_{0}^{\ld\,\,t^{-1}}dw\,\,\arctan(w)}\\
&=\frac{1}{\pi}\pthesis{\frac{\pi}{2}\ld t^{-1}-\pint{\ld\,t^{-1}\,\arctan(\ld t^{-1})-\frac{1}{2}\ln\pthesis{1+\frac{\ld^2}{t^2}}}}\\
&=\frac{1}{\pi}\ln \pthesis{\frac{1}{t}}+\frac{1}{\pi}\pthesis{\ld\,t^{-1} \,\arctan\pthesis{\frac{t}{\ld}}+
\frac{1}{2}\ln\pthesis{t^{2}+\ld^{2}}}.
\end{align*}
Therefore, it follows from the above expression that
\begin{equation*}
\myH_{\Omega, \Omega^c}^{(1)}(t)=\frac{2}{\pi}\,t\ln\pthesis{\frac{1}{t}}
+\frac{2}{\pi}\pthesis{\ld \,\arctan\pthesis{\frac{t}{\ld}}+
\frac{1}{2}\,\,t\,\ln\pthesis{t^{2}+\ld^{2}}}, 
\end{equation*}
for all $t>0$ and this completes the proof of  part $(ii)$ of 
Theorem \ref{mth}.
\bigskip

{\bf Case $0<\alpha<1$}: Assume $0<t \leq \min\set{\ld^{\alpha}, e^{-1}}$.  In \cite[p.~88]{Sato}, the following power series representation is provided for the one--dimensional density function $p^{(\alpha)}_1(z)$ for any  $z>0$, $0<\alpha <1$.
\begin{align*}
p_1^{(\alpha)}(z)=\mysum{n}{1}{\infty}a_n(\alpha)z^{-1-n\alpha}
\end{align*}
with
$$a_n(\alpha)=(-1)^{n-1}\,\frac{\Gamma\pthesis{n\alpha+1}}{\pi\,n!}\sin\pthesis{\frac{\pi \,n\, \alpha}{2}}
.$$

Notice that by applying  Fubini's Theorem, we obtain for $w>0$
\begin{equation}\label{absconv}
\int_{w}^{\infty}dz\,\pthesis{\mysum{n}{1}{\infty}\abs{a_{n}(\alpha)}z^{-1-n\alpha}}=\mysum{n}{1}{\infty}\frac{\abs{a_{n}(\alpha)}}{n\alpha}\pthesis{\frac{1}{w^\alpha}}^n.
\end{equation}
By appealing to  the Stirling's formula
\begin{align}\label{asym1}
\lim\limits_{ n\rightarrow \infty}\frac{\Gamma(t_n+1)} { \sqrt{2\pi t_n}(t_ne^{-1})^{t_n}}=1,
\end{align}
for any increasing sequence $\set{t_n}_{n\in \mathbb{N}}$ of positive number converging to infinity, we can  prove that the series on the right hand side of the equality \eqref{absconv} is convergent for all $w>0$ since by the root test its radius of convergence is infinity. To see this, we note that
\begin{align}\label{ub}
\abs{a_{n}(\alpha)}\leq \frac{\Gamma\pthesis{n\alpha+1}}{n!},
\end{align}
which implies 
$$\varlimsup_{n\rightarrow \infty}\pthesis{\frac{\abs{a_n(\alpha)}}{n}}^{1/n}\leq \varlimsup_{n\rightarrow \infty}\pthesis{\frac{\Gamma\pthesis{n\alpha+1}}{n\,n!}}^{1/n}.$$
Thus, by using that $\Gamma(n+1)=n!$ and \eqref{asym1} (with $t_n=n\alpha$  and $t_n=n!$), we arrive at
\begin{align}\label{lim}
\varlimsup_{n\rightarrow \infty}\pthesis{\frac{\Gamma\pthesis{n\alpha+1}}{n\,n!}}^{1/n}&=\varlimsup_{n\rightarrow \infty}\pthesis{\frac{\sqrt{2\pi n\alpha}(n\alpha e^{-1})^{n\alpha}}{n\,\sqrt{2\pi n}(n e^{-1})^{n}}}^{1/n}\\ \nonumber
&=\varlimsup_{n\rightarrow \infty}
\pthesis{\frac{\sqrt{\alpha}}{n}}^{1/n}\frac{\alpha^{\alpha}e^{1-\alpha}}{n^{1-\alpha}}=0,
\end{align}
whenever $0<\alpha<1$.
Therefore, by using once more Fubini's Theorem, we have for $w>0$
\begin{align}\label{intsum1}
\int_{w}^{\infty}dz\,p_1^{(\alpha)}(z) =\mysum{n}{1}{\infty}\frac{a_{n}(\alpha)}{n\alpha}\pthesis{\frac{1}{w^\alpha}}^n.
\end{align}

Next, it is easy to show that
\begin{align}\label{ic}
\int_{1}^{\ld t^{-\inalp}}dw\,\int_{w}^{\infty}dz\,z^{-1-n\alpha}&= \pthesis{n\ln\pthesis{\frac{1}{t}} + \ln\pthesis{\ld}} \cdot \mathbbm{1}_{\set{n\alpha\,=1}}\\  \nonumber
&+\pthesis{\frac{\ld^{1-n\alpha}\,\,t^{n-\inalp}-1}{n\alpha\,\,(1-n\alpha)}}\cdot\mathbbm{1}_{\set{n\alpha\neq 1}}.
\end{align}

Before continuing, let us introduce some notation to simplify 
the formulas to appear below. For $m\in \mathbb{N}\cup\set{\infty}$, $t>0$
and $1/\alpha \notin \mathbb{N}$, we set 
\begin{align}\label{ps}
s_{m}(t)&=\mysum{n}{1}{m}\frac{a_n(\alpha)\ld^{1-n\alpha}\,t^{n}}{n\alpha (1-n\alpha)}, \,\,\,\,\,
r_{m}(t)=\mysum{n}{1}{m}\frac{a_n(\alpha)\,t^{n}} {n\alpha (1-n\alpha)}, \\ 
\tilde{s}_{m}(t)&=\mysum{n}{m}{\infty}\frac{a_n(\alpha)
\ld^{1-n\alpha}\,t^{n}}{n\alpha (1-n\alpha)},\,\,\,\,\,  \nonumber\tilde{r}_{m}(t)=\mysum{n}{m}{\infty}\frac{a_n(\alpha)\,t^{n}} {n\alpha (1-n\alpha)}.
\end{align}
 
These series are absolutely convergent for all $t>0$  since by using \eqref{ub} and \eqref{lim}, we obtain that
\begin{align*}
\varlimsup_{n\rightarrow\infty}
\pthesis{
\frac{ \abs{a_n(\alpha)}\ld^{1-n\alpha}}{n\alpha 
\abs{1-n\alpha}}}
^{1/n}=\varlimsup_{n\rightarrow\infty}
\pthesis{
\frac{ \abs{a_n(\alpha)}}{n\alpha 
\abs{1-n\alpha}}}
^{1/n}=0
\end{align*}
for all $0<\alpha<1$ and $1/\alpha \notin \mathbb{N}$.  We remark that $\tilde{s}_{m}(t)$ being absolutely convergence for all $t>0$ implies that if $1/\alpha\neq n$  for all $n\geq m$, then $M_{\alpha,m}=\sup\set{\frac{ \abs{a_n(\alpha)}\ld^{1-n\alpha}}{n\alpha 
\abs{1-n\alpha}}: n\geq m}<\infty$ and for all $0<t<\min\set{\ld^{\alpha}, e^{-1}}$, we  arrive at
\begin{align}\label{ineqfs}
\abs{\tilde{s}_{m}(t)}\leq M_{\alpha,m}\mysum{n}{m}{\infty}t^n=\frac{ M_{\alpha,m}\pthesis{\min\set{\ld^{\alpha}, e^{-1}}}^m}{1-\min\set{\ld^{\alpha}, e^{-1}}}.
\end{align}

As a result of the preceding  facts and the elementary tools
of calculus, we are allowed to interchange in \eqref{intsum1} the sum  with the integral sign over any compact set contained in $(0,\infty)$. Thus, if  $1/\alpha \notin \mathbb{N}$, we conclude by \eqref{ic} and \eqref{ps} that
\begin{align}\label{1expr}
\int_{1}^{\ld t^{-\inalp}}dw\,\int_{w}^{\infty}dz\,p_1^{(\alpha)}(z)
&=
 t^{-\inalp}\,\,s_{\infty}(t)-r_{\infty}(1)\\ \nonumber 
&=t^{-\inalp}\,\,s_{\pint{\inalp}}(t)- r_{\infty}(1)+t^{-\inalp}\,\,\tilde{s}_{\pint{\inalp}+1}(t),
\end{align}
where $\left[1/\alpha\right]$ denotes the integer part of $1/\alpha$.
 On the other hand, if $\alpha=1/N$ for some $N \in \mathbb{N}$, we obtain
\begin{align}\label{2expr}
\int_{1}^{\ld t^{-N}}dw\,\int_{w}^{\infty}dz\,p_1^{(1/N)}(z)
&= \,t^{-N}\,s_{N-1}(t)-r_{N-1}(1)+ 
   a_N(1/N)N\, \ln \pthesis{\frac{1}{t}}\\ \nonumber&+a_{N}(1/N)\ln(\abs{\Omega})+\,t^{-N}\,\tilde{s}_{N+1}(t)-\tilde{r}_{N+1}(1) \\\nonumber
   &=\,t^{-N}\,s_{N-1}(t)+
   a_N(1/N)N\, \ln \pthesis{\frac{1}{t}}+C_{N}^*(\Omega)+t^{-N}\tilde{s}_{N+1}(t).
   \nonumber
\end{align}
where
\begin{align}\label{const1}
C_{N}^*(\Omega)=a_{N}\pthesis{1/N}\ln(\ld)-
r_{N-1}(1)-\tilde{r}_{N+1}(1).
\end{align}

We rewrite $\ele$ given in \eqref{eldef} as follows.
\begin{align*}
\ele=\int_{0}^{1}dw\,\Prob\pthesis{w\leq X_1}+
\int_{1}^{\ld t^{-\inalp}}dw\,\int_{w}^{\infty}dz\,p_1^{(\alpha)}(z).
\end{align*}
Then, by using the last equality and the identities
\eqref{1expr} and \eqref{2expr},  we arrive at
\begin{align*}
\Hset{\Omega}{\Omega^c}=2\,s_{\pint{\inalp}}(t)+ C_{\alpha}\,t^{\inalp}+ 2\, \tilde{s}_{\pint{\inalp}+1}(t)
\end{align*}
for $1/\alpha \notin \mathbb{N}$. Here, $$C_{\alpha}=2\pthesis{\int_{0}^{1}dw\,\, \Prob\pthesis{w\leq X_1}-
r_{\infty}(1)}.$$
As for the case $\alpha=1/N$, some $N\in \mathbb{N}$, we have
\begin{align*}
\myH_{\Omega,\Omega^c}^{\pthesis{1/N}}(t)=2\,\,s_{N-1}(t)+2N\,a_N(1/N)\, t^{N}\ln\pthesis{\frac{1}{t}}
+C_{N}(\Omega)\,t^{N}+2\,\tilde{s}_{N+1}(t)
\end{align*}
with $$C_{N}(\Omega)=2\pthesis{\int_{0}^{1}dw\,\Prob\pthesis{w\leq X_1}+ C_{N}^*(\Omega)}$$ and $C_N^*(\Omega)$ as defined in \eqref{const1}. Hence, the proof of part $(iii)$ in
Theorem \ref{mth} is complete by taking
$R_{\alpha}(t)=2\,\tilde{s}_{\pint{\inalp}+1}(t)$ and observing that inequality \eqref{ineqfs} yields
$$\abs{R_{\alpha}(t)}\leq C\,t^{\pint{\inalp}+1},\,\,\,\,
0<t \leq \min\set{\ld^{\alpha}, e^{-1}},$$
for  some $C>0$.

\section{proof of theorem \ref{2thm}}\label{sec:T2}
We start by recalling equation \eqref{dbysub} which allows us to write the transition densities $\palp(x,y)$ by subordination of the Gaussian kernel.  Therefore, an application of  Fubini's Theorem yields
\begin{align}\label{alpsem}
\myH_{\Omega,\Omega^c}^{(\alpha)}(t)&=
\int_{\Omega}dx\,\int_{\Omega^c}dy\,\palp(x-y)\\ \nonumber
&=\int_{\Omega}dx\,\int_{\Omega^c}dy\, \E{p^{(2)}_{S_t}(x-y)} 
=\E{\myH_{\Omega,\Omega^c}^{(2)}(S_{t})}.
\end{align}

{\bf Proof of part $(i)$:} Assume $1<\alpha<2$. With the aid of the inequality \eqref{isop} which is valid for all positive time,  equality \eqref{alpsem}, the fact that $S_t\eid t^{2/\alpha}S_1$ and formula \eqref{gen.expt.S1}, it easily follows that
\begin{align}
\myH_{\Omega,\Omega^c}^{(\alpha)}(t)\leq 
\frac{ \mathcal{H}^{d-1}(\bd)}{\sqrt{\pi}} \E{S_{t}^{1/2}}=
\frac{1}{\pi}\,\,\Gamma\pthesis{1-\inalp}\,\,\mathcal{H}^{d-1}(\partial \Omega)\,\, t^{\inalp}
\end{align}
for all $t>0$ and with this we have proved  \eqref{iso2}.

On the other hand, by using once more  \eqref{alpsem}
and $S_t\eid t^{2/\alpha}S_1$, we obtain
\begin{align}\label{StableLes}
\frac{\myH_{\Omega,\Omega^c}^{(\alpha)}(t)}{t^{\inalp}\mathcal{H}^{d-1}(\partial \Omega)}&=
\int_{0}^{\infty}ds\,\,
\frac{ \myH_{\Omega,\Omega^c}^{(2)}(t^{2/\alpha}s)}{ t^{\inalp}\mathcal{H}^{d-1}(\partial \Omega)}\,\,\eta^{(\alpha/2)}_1(s)  \\ 
&=\int_{0}^{\infty}ds\,\,G(s,t),\nonumber
\end{align}
with
$$G(s,t)=s^{1/2}
\pthesis{\frac{\myH_{\Omega,\Omega^c}^{(2)}(s\,t^{2/\alpha})}{\pthesis{s\,t^{2/\alpha}}^{1/2}
\mathcal{H}^{d-1}(\partial \Omega)}}\eta^{(\alpha/2)}_1(s).
$$
We now observe two facts. First, by \eqref{hc2} we have
\begin{align*}
\lim\limits_{\tgo}G(s,t)=\frac{1}{\sqrt{\pi}}\,\,s^{1/2}\,\,\eta^{(\alpha/2)}_1(s).
\end{align*}
Secondly, by \eqref{isop}
\begin{align*}
0\leq G(s,t)\leq \frac{1}{\sqrt{\pi}}\,\,s^{1/2}\,\,\eta^{(\alpha/2)}_1(s)
\end{align*}
for all $t,s>0$ with $s^{1/2}\,\,\eta^{(\alpha/2)}_1(s)\in L^{1}((0,+\infty))$ because of \eqref{gen.expt.S1}. Hence,  the assertion \eqref{Isoplim} is an easy consequence of the Lebesgue Dominated Convergence Theorem and the  identity \eqref{StableLes}.\qed

We now continue with the proof of  $(ii)$ of Theorem \ref{2thm}.  This requires a much more delicate approach.  In order to make this presentation as clear as possible, we devote the next section to it. 

\section{cauchy processes in higher dimension}\label{sec:Cauchy}
In this section, we will adapt the  techniques used in
\cite{Mir} for the Gaussian heat kernel. This requires some 
additional considerations since as we have already 
pointed out the  Gaussian kernel has an exponential decay whereas the Cauchy heat kernel
$ p^{(1)}_t(x,y)$ defined in \eqref{Cauchyk} has a polynomial decay. 

From now on, we write every vector $x\in \Rd$ as 
$x=(\bar{x},x_{d})$
with $\bar{x}=(x_1,...,x_{d-1})\in \R^{d-1}.$
\begin{lem}\label{const}
For all integer $d\geq 2$, wet set
$$\gamma_d=\frac{\Gamma\pthesis{\frac{d+1}{2}}}{\pi^{\frac{d+1}{2}}}
\int_{\R^{d-1}}\frac{dw}{(1+\abs{w}^2)^{\frac{d+1}{2}}}.$$
\end{lem}
Then, $\gamma_d=\frac{1}{\pi}$.
\begin{proof}
By appealing to polar coordinates, we have
\begin{align*}
\gamma_d=\frac{\Gamma\pthesis{\frac{d+1}{2}}}{\pi^{\frac{d+1}{2}}}\cdot \frac{2\pi^{\frac{d-1}{2}}}{\Gamma\pthesis{\frac{d-1}{2}}}
\int_{0}^{\infty}\frac{r^{d-2}}{
\pthesis{1+r^2}^{\frac{d+1}{2}}}dr.
\end{align*}
Next, the properties of the gamma function and the change of variables $r=\tan(\theta)$ yield
\begin{align*}
\gamma_d=\frac{d-1}{\pi}\int_{0}^{\frac{\pi}{2}}\sin^{d-2}(\theta)\cos(\theta)d\theta=\frac{1}{\pi}.
\end{align*}
\end{proof}

\begin{lem}\label{ml}
Let $H=\set{(\bar{x},x_d)\in \Rd: x_d<0}$ and $\delta, \varepsilon>0$.  Set
$H^{\delta}=\R^{d-1}\times (0,\delta)$ and $H_{\varepsilon}=
\R^{d-1}\times (-\varepsilon,0)$. Assume
$\varphi\in C^{1}_c(\Rd)$ and consider the compact set
\begin{align}\label{kdef}
K=\set{\bar{x}\in \R^{d-1}: \exists \, x_d  \in \R \,\,such \,\, that\,
\, (\bar{x},x_d)\in supp(\varphi) }.
\end{align}
Then, there exists a function $R(t)$
such that
\begin{align}\label{Cauchyplane}
\int_{H^{\delta}}dx\,\varphi(x)\int_{H_{\varepsilon}}dy\,p^{(1)}_t(x,y)=
\frac{1}{\pi} \pthesis{\int_{K}d\bar{x}\,\varphi(\bar{x},0)}
t\ln\pthesis{\frac{1}{t}}+R(t),
\end{align}
with 
\begin{align}\label{RCauchy}
\abs{R(t)}\leq C_{\varepsilon,\delta,\varphi}\,\,t
\end{align}
for all $0<t<e^{-1}$.
\end{lem}

\begin{proof}
We first note that  the integral on the left hand side 
of \eqref{Cauchyplane} equals 
\begin{align*}
\int_{H^{\delta}}dx\,\,\varphi(x)\int_{-\varepsilon}^{0}dy_d
\int_{\R^{d-1}}\frac{\Gamma\pthesis{\frac{d+1}{2}}}{\pi^{\frac{d+1}{2}}}\cdot\frac{\,t\,d\bar{y}}{\pthesis{t^{2}+\abs{x_d-y_d}^{2}+
\abs{\bar{x}-\bar{y}}^2}^{(d+1)/2}}.
\end{align*}
By considering the change of variable
$$\bar{y}=\bar{x}-\sqrt{t^{2}+
\abs{x_d-y_d}^2}\,\,w$$ and Lemma \ref{const},
we reduce the last integral to 
\begin{align*}
\frac{t}{\pi}\int_{\R^{d-1}}d\bar{x}\int_{0}^{\delta}dx_d \,\,\varphi(\bar{x},x_d)
\int_{-\varepsilon}^{0}\frac{\,dy_d}{t^{2}+
\abs{x_d-y_d}^2}.
\end{align*}
Thus,  by making the new change of variables $x_d-y_d=tz$ in the last integral expression, we arrive at
\begin{align}\label{Cauchyplane2}
\int_{H^{\delta}}dx\,\varphi(x)\int_{H_{\varepsilon}}dy\,p^{(1)}_t(x,y)
=\frac{1}{\pi}\int_{\R^{d-1}}d\bar{x}
\int_{0}^{\delta}dx_d\,\,\varphi(\bar{x},x_d)\,\,F_t(x_d,\varepsilon)
\end{align}
with
\begin{align}\label{Fdef}
F_t(x_d,\varepsilon)&=\arctan\pthesis{\frac{x_d+\varepsilon}{t}}-
\arctan\pthesis{\frac{x_d}{t}}\nonumber\\
&=\arctan\pthesis{\frac{t}{x_d}}-\arctan\pthesis{\frac{t}{x_d+\varepsilon}}.
\end{align}

Let us set at this point
\begin{align}\label{hdef}
h(\bar{x},x_d)=\varphi(\bar{x},x_d)-\varphi(\bar{x},0).
\end{align}
Notice that according to \eqref{kdef}, we have
\begin{align}\label{fact1} 
\varphi(\bar{x},x_d)=h(\bar{x},x_d)=0
\end{align}
for all $(\bar{x},x_d)\in K^c\times \R$. 
Since $\varphi(x)$ is compactly supported with continuous partial derivatives it follows from the Taylor expansion
that
\begin{align}\label{hlip}
\abs{h(\bar{x},x_d)}=\abs{\int_{0}^{1}
\nabla \varphi((\bar{x},x_d)-s(\bar{x},0))\cdot(0,x_d)ds}\leq
||\nabla \varphi||_{\infty}|x_d|.
\end{align}

We next consider  the continuous function $\Pi:\R^d\rightarrow \R^{d-1}$ defined by $\Pi(\bar{x},x_d)=\bar{x}.$ Then 
$$K=\set{\Pi(\bar{x},x_d):(\bar{x},x_d)\in supp(\varphi)}.$$ 
Thus, because of the continuity of $\Pi$, $K$  is a compact  set in $\R^{d-1}$  whose finite Lebesgue measure will be denoted in what follows by $|K|$. 

Now, by appealing to \eqref{Cauchyplane2},  \eqref{hdef} and
\eqref{fact1},  we find that
\begin{align}\label{Cauchyp3}
\int_{H^{\delta}}dx\,\varphi(x)\int_{H_{\varepsilon}}dy\,p^{(1)}_t(x,y)
=\frac{1}{\pi}\int_{K}d\bar{x}
\int_{0}^{\delta}\,dx_d\,\varphi(\bar{x},0)\,F_t(x_d,\varepsilon)+
R_2(t)
\end{align}
with
\begin{align}\label{R2}
R_2(t)=\frac{1}{\pi}\int_{K}d\bar{x}
\int_{0}^{\delta}dx_d\,h(\bar{x},x_d)\,F_t(x_d,\varepsilon).
\end{align}
As for the first integral term on the right hand side of the equation \eqref{Cauchyp3}, we have by using  the elementary  identities \eqref{ei1} and \eqref{ei2} that it is equal to
\begin{align*}
\frac{1}{\pi} \pthesis{\int_{K}d\bar{x}\,\varphi(\bar{x},0)}
\pthesis{t\ln\pthesis{\frac{1}{t}}+R_1(t)},
\end{align*}
with
\begin{align*}
R_1(t)&=
\varepsilon \arctan\pthesis{\frac{t}{\varepsilon}}+
\delta\arctan\pthesis{\frac{t}{\delta}}\\
&-(\delta+\varepsilon)\arctan\pthesis{\frac{t}{\delta+\varepsilon}}+
\frac{t}{2}\ln\pthesis{\frac{(t^2+\varepsilon^2)(t^{2}+
\delta^2)}{t^{2}+(\delta+\varepsilon)^2}}.
\end{align*}
We remark that due to the inequality $\arctan(x)\leq x$ for
$x>0$ and the fact that the function $\ln\pthesis{\frac{(t^2+\varepsilon^2)(t^{2}+
\delta^2)}{t^{2}+(\delta+\varepsilon)^2}}$ is continuous  for  $0\leq t\leq e^{-1}$ because $\delta$ and $\varepsilon$  are positive, we obtain that $\abs{R_1(t)}\leq C_{\delta,\varepsilon}\,t$ for  some $C_{\delta,\varepsilon}>0$.

As for $R_2(t)$, we first observe that by \eqref{Fdef}, 
\begin{equation}\label{Fub}
0\leq F_{t}(x_d,\varepsilon)\leq \arctan\pthesis{\frac{t}{x_d}}\leq \frac{t}{x_d}.
\end{equation}
Therefore, by combining \eqref{hlip}, \eqref{R2}  and \eqref{Fub}, we have 
\begin{align*}
\abs{R_2(t)}\leq \frac{1}{\pi} \int_{K}d\bar{x}\int_{0}^{\delta}
dx_d\abs{h(\bar{x},x_d)}F_{t}(x_d,\varepsilon)\leq \frac{1}{\pi} \,\delta\,||\nabla \varphi||_{\infty} |K|\,t.
\end{align*}

Now by setting $R(t)=\frac{1}{\pi}\pthesis{\int_{K}d\bar{x}\,\varphi(\bar{x},0)}R_1(t)+R_{2}(t)$ and putting together all the estimates given above we conclude \eqref{RCauchy} and this finishes the proof of Lemma \ref{ml}.
\end{proof}

Before proceeding, we comment further on the last result. In probabilistic terms, we have 
$$\int_{H^{\delta}}dx\,\,\varphi(x)\int_{H_{\varepsilon}}dy\,\,p^{(1)}_t(x,y)=\int_{H^{\delta}}dx\,\,\varphi(x)\,\,\Prob^{x}\pthesis{
X_t\in H_{\varepsilon}}.$$

The goal of the last integral is to understand how the paths of the 
Cauchy process ``perceive"  the boundary of the half plane $H$. The above lemma says
that when  $\varphi(x)\in C^{1}_c(\R^{d})$,
the process  ``feels" the influence of the boundary $\partial H=\R^{d-1}\times\set{0}$ by means of the
term
$$\int_{K\subset \R^{d-1}}d\bar{x}\,\varphi(\bar{x},0).$$

For a bounded domain with smooth boundary $\Omega$, the paths conditioned to start in $\Omega$  and exit at time $t$ should ``view" the boundary as a half-plane. Therefore, it is expected that we can replace
$\int_{K}d\bar{x}\, \varphi(\bar{x},0)$ with
$\int_{\bd}\varphi(\sigma)\,\,d\mathcal{H}^{d-1}(\sigma).$ To this aim, we recall  some definitions and geometric properties on uniformly $C^{1,1}$-regular domains.  We refer the reader to \cite{Mir}, \cite{Wloka} and references therein for details and further  considerations on the topic.

We set
$\rho_{\Omega}(x)=\inf\set{\abs{x-\sigma}:\sigma\in \bd}$ and for  $\delta,\varepsilon>0$, we define
\begin{align}\label{inout}
\Omega^{\delta}&=\set{x\in \Omega^c:\md<\delta},\\
\Omega_{\varepsilon}&=\set{x\in \Omega\,\,\,:\md<\varepsilon}.
\nonumber
\end{align}  

\begin{prop}\label{Cdomprop}
Let $\Omega \subset \Rd$ be a uniformly $C^{1,1}$-regular  bounded domain. Then, 
\begin{enumerate}
\item[$(a)$] there exists $\varepsilon, \delta > 0$ such that the maps
\begin{align*}
J:\bd \,\times \, [0, \delta] \rightarrow \Omega^{\delta},
\,\,\,J(\sigma, r) =\sigma + r\,\nu(\sigma),
\end{align*}
\begin{align}\label{bdrep}
\tilde{J}:\bd \,\times \, [0, \varepsilon] \rightarrow \Omega_{\varepsilon},\,\,\, \tilde{J}(\sigma, r) = \sigma - r\,  \nu(\sigma),
\end{align}
where $\nu(\sigma)$ is the outward unit normal to $\bd$ at $\sigma$, are $C^{1,1}$-diffeomorphisms.
\\

\item[$(b)$] Given $\eta > 0$, there exists  a  finite covering 
$V =\set{V_i}$ of $\bd$ and $C^{1,1}$--diffeomorphisms
$\psi_i : K_i \rightarrow V_i,$ with $K_i$ open subset of $\R^{d-1}$ such that if we set 
$$\Psi_i (\xi,\rho )= \psi_i (\xi) + \rho\,\nu(\psi_i (\xi)),\,\,\,\,\,\,\,
\xi \in K_i, \,\,\rho \in (-\varepsilon, \delta),$$ then the family
of open sets  $U = \set{U_i}$ with
$U_i = \Psi_i\pthesis{K_i \times (-\varepsilon, \delta)}$
covers $\Omega_{\varepsilon}\cup \Omega^{\delta}$ with Jacobians satisfying
\begin{align}\label{detcomp}
\abs{D\Psi_i(\xi,\rho)}&=1+\mathcal{O}(\eta),
\,\,
\xi \in K_i, \,\,\rho \in (-\varepsilon, \delta),
\\
|D\Psi_i^{-1} (x)| &= 1 + \mathcal{O}(\eta),\,\,x\in U_i\nonumber \\
|D\psi_i^{-1} (x)| &= 1 + \mathcal{O}(\eta), \,\,x \in V_i.\nonumber
\end{align}
Also
\begin{align}\label{distance}
\abs{\Psi_i(z,r)-\Psi_i(\xi,\rho)}^2=\abs{(z,r)-(\xi,\rho)}^2
\pthesis{1+\mathcal{O}(\eta)},
\end{align}
for all $\xi,z \in K_i$ and $\rho, r \in (-\varepsilon, \delta)$. Here, 
we use the notation $\mathcal{O}(\eta)$ to mean a function which is upper bounded in absolute value by $C\eta$, where the constant $C$  depends only on $\Omega, \varepsilon, \delta$. 
\end{enumerate}
\end{prop}

The main result of this section is the following.
\begin{thm}\label{2mth}
Let $\Omega\subset\Rd$ be a uniformly $C^{1,1}$-regular  bounded  domain.
Consider $\Omega_{\varepsilon}$ and $\Omega^{\delta}$ the
inner and outer tubular neighbourhoods of $\bd$ defined in
\eqref{inout}. Then, for every $\varphi\in C_c^{1}(\Rd)$ we have
\begin{align}
\lim\limits_{\tgo}\frac{1}
{t\ln\pthesis{\frac{1}{t}}}\int_{\Omega^{\delta}}\,dx\,\varphi(x)
\int_{\Omega_{\varepsilon}}\,dy&\,p^{(1)}_t(x,y)
=
\frac{1}{\pi}\int_{\bd}\varphi(\sigma)\,d\mathcal{H}^{d-1}(\sigma).
\end{align}
\end{thm}

\begin{proof}
Let $\eta>0$ and consider the finite family of open sets $U=\set{U_i}$ provided by part $(b)$ in the last proposition.  Now, let $\set{\chi_i}$ be a smooth partition of the unity subordinated to the covering $U$ (see \cite[Theorem 1.2]{Wloka}). We assume without loss of generality that 
$supp( \chi_i)\subset U_i $. Therefore, by using the fact that $\sum\limits_{i} \chi_i(x)=1$ for every $x \in \cup U_i $, we have 
\begin{align*}
\int_{\Omega^{\delta}}\,dx\,\varphi(x)
\int_{\Omega_{\varepsilon}}\,dy\,p^{(1)}_t(x,y)&=
\int_{\Omega^{\delta}}\,dx\,\varphi(x)\pthesis{\sum\limits_{i} \chi_i(x)}\int_{\Omega_{\varepsilon}}\,dy\,p^{(1)}_t(x,y)\\ \nonumber&=\sum\limits_{i}\int_{\Omega^{\delta}}\,dx\,\varphi(x) \chi_i(x)\int_{\Omega_{\varepsilon}}\,dy\,p^{(1)}_t(x,y) \\ \nonumber
&=\sum\limits_{i}\pthesis{I_{i}+\tilde{I}_{i}}
\end{align*}
with 
\begin{align*}
I_{i}=\int_{\Omega^{\delta}\cap supp(\chi_i)}\,dx\,\varphi(x)\chi_i(x)
\int_{\Omega_{\varepsilon}\cap U_i}\,dy&\,p^{(1)}_t(x,y),\\
\tilde{I}_{i}=\int_{\Omega^{\delta}\cap supp(\chi_i)}\,dx\,\varphi(x)\chi_i(x)
\int_{\Omega_{\varepsilon}\setminus U_i}\,dy&\,p^{(1)}_t(x,y).
\end{align*}
 
 Observe that $supp(\chi_i)\subset U_i$ is compact and also disjoint from the compact set $\overline{\Omega_{\varepsilon}\setminus U_i}$, then
\begin{align*}
 \inf\set{|x-y|:x\in supp(\chi_i),y \in \overline{\Omega_{\varepsilon}\setminus U_i}}=\mu_i>0.
\end{align*}
Thus, by appealing to the explicit form of the Cauchy heat kernel \eqref{Cauchyk} and the fact $0\leq \chi_i\leq1$ for every $i$,  we conclude 
\begin{align*}
\lim\limits_{\tgo}\abs{\frac{1}{t\ln\pthesis{\frac{1}{t}}}\sum\limits_{i}\tilde{I}_i}\leq
\lim\limits_{\tgo}\frac{k_d}{\ln\pthesis{\frac{1}{t}}}\pthesis{\sum\limits_{i}\frac{|\Omega_{\varepsilon}\setminus U_
i|}{\mu_i^{d+1}\,}
\int_{\Omega^{\delta}\cap supp(\chi_i)} dx\,\,\abs{\varphi(x)}}=0.
\end{align*}

Now, we proceed to deal with the term $I_i$. We start by expressing  every $x\in \Omega^{\delta}\cap supp(\chi_i)$  and 
$y \in \Omega_{\varepsilon} \cap U_i$ under the new variables introduced in Proposition \ref{Cdomprop}. Namely, 
\begin{align*}
y &= \Psi_i(z, r),\,\,\,\,\, z \in K_i,\, r \in [-\varepsilon, 0], \\
x &= \Psi_i(\xi,\rho ),\,\,\,\,\, \xi\in K_i,\, \rho 
\in [0, \delta].
\end{align*}
Then, using these equalities, we obtain
\begin{align*}
I_i =\int_{K_i\times(0,\delta)}d\xi\,d\rho\,\chi_i\pthesis{\Psi_i(\xi,\rho)} \varphi\pthesis{\Psi_i(\xi,\rho)}
\int_{K_i\times(-\varepsilon,0)}dz\,dr\,
p_t((z,r),(\xi,\rho)),
\end{align*}
where we have set
\begin{align*}
p_t((z,r),(\xi,\rho))=
p_{t}^{(1)}(\Psi_i(z,r),\Psi_i(\xi,\rho))\abs{D\Psi_i(z,r)}\abs{D\Psi_i(\xi,\rho)}.
\end{align*}
Define $g_t(x,y)=\frac{|x-y|^2}{t^2+|x-y|^2}$ with $x,y\in \R^{d}$ and $t>0$. Hence, by using the estimates given in \eqref{detcomp} and \eqref{distance}, we find that
\begin{align*}
p_t((z,r),(\xi,\rho))&=p_t^{(1)}((z,r),(\xi,\rho))
\left\lbrack\frac{1+\mathcal{O}(\eta)}{\pthesis{1+g_t((\xi,\rho), (z,r))\mathcal{O}(\eta)}^{(d+1)/2}}\right\rbrack.
\end{align*}
We now observe by using that $0\leq  g_t\leq 1$ and the above expression, we can chose  $\eta$  very small but arbitrary such that
\begin{align}
p_t((z,r),(\xi,\rho))=p_t^{(1)}((z,r),(\xi,\rho))
\pthesis{1+\mathcal{O}(\eta)}.
\end{align}
Therefore, we conclude by Proposition \ref{Cdomprop} and \eqref{detcomp} that
\begin{align}\label{lim1}
\lim\limits_{\tgo}\frac{1}{t\ln\pthesis{\frac{1}{t}}}\sum\limits_{i}I_i&=
\pthesis{1+\mathcal{O}(\eta)}\frac{1}{\pi}\sum \limits_{i}
\int_{K_i}\chi_{i}\pthesis{\Psi_i(\xi,0)}\varphi\pthesis{\Psi_i(\xi,0)}
d\xi\\  \nonumber
&=
\pthesis{1+\mathcal{O}(\eta)}\frac{1}{\pi}\sum \limits_{i}
\int_{K_i}\chi_{i}\pthesis{\psi_i(\xi)}\varphi\pthesis{\psi_i(\xi)}
d\xi\\ \nonumber
&=\pthesis{1+\mathcal{O}(\eta)}\frac{1}{\pi}\sum \limits_{i}
\int_{V_i}\chi_{i}\pthesis{\sigma}\varphi\pthesis{\sigma}
\abs{D\psi_i^{-1}(\sigma)}d\mathcal{H}^{d-1}(\sigma).
\end{align}
Notice that $V_i=U_i\cap \partial \Omega$. It follows from the facts $|D\psi_i^{-1} (\sigma)| = 1 + \mathcal{O}(\eta), \,\,\sigma\in V_i$   and the assumption $supp( \chi_i)\subset U_i $ that 
\begin{align*}
\int_{V_i}\chi_{i}\pthesis{\sigma}\varphi\pthesis{\sigma}
\abs{D\psi_i^{-1}(\sigma)}d\mathcal{H}^{d-1}(\sigma)&=
\pthesis{ 1 + \mathcal{O}(\eta)}\int_{\partial \Omega}\chi_{i}\pthesis{\sigma}
\mathbbm{1}_{ supp(\chi_i)}(\sigma)
\varphi(\sigma)
d\mathcal{H}^{d-1}(\sigma)\\ \nonumber &=
\pthesis{ 1 + \mathcal{O}(\eta)}\int_{\partial \Omega}\chi_{i}\pthesis{\sigma}
\varphi(\sigma)
d\mathcal{H}^{d-1}(\sigma).
\end{align*}
Thus, the last expression and equality \eqref{lim1} allow us to conclude
\begin{align*}
\lim\limits_{\tgo}\frac{1}{t\ln\pthesis{\frac{1}{t}}}\sum\limits_{i}I_i&=\pthesis{1+\mathcal{O}(\eta)}\frac{1}{\pi}
\int_{\bd}\pthesis{\sum \limits_{i}\chi_{i}(\sigma)}\varphi\pthesis{\sigma}
d\mathcal{H}^{d-1}(\sigma) \\
&=\pthesis{1+\mathcal{O}(\eta)}\frac{1}{\pi}
\int_{\bd}\varphi\pthesis{\sigma}
d\mathcal{H}^{d-1}(\sigma),
\end{align*}
where we have used that $\sum\limits_{i} \chi_i(\sigma)=1$ for every $\sigma \in \cup U_i $ and $\partial \Omega \subset  \cup U_i$.
The proof is complete by letting $\eta$ go to zero.
\end{proof}

\begin{rmk} 
Let $\Omega\subset\Rd$ be a uniformly   $C^{1,1}$-regular bounded domain  and $\varepsilon, \delta$ as given in Proposition \ref{Cdomprop}. It is clear because of  the boundedness of
$\Omega$ that $\overline{\Omega^{\delta}\cup
\Omega_{\varepsilon}}$ is contained in some open ball.  Thus, by
Corollary 1.2 in \cite[p.~8]{Wloka}, there exists  an infinitely differentiable and compactly supported function $\varphi$ such that
$$\overline{\Omega^{\delta}\cup \Omega_{\varepsilon}}\subset \set{ x \in supp(\varphi): 
\varphi(x)=1}.$$
Therefore, as an application of Theorem \ref{2mth}, we conclude
\begin{align*}
\lim\limits_{\tgo}\frac{\myH^{(1)}_{\Omega_{\varepsilon}, \Omega^{\delta}}(t)}
{t\ln\pthesis{\frac{1}{t}}}
=\frac{1}{\pi}\, \mathcal{H}^{d-1}(\bd).
\end{align*}
\end{rmk}

We observe that for every $\delta, \varepsilon>0$, we have
\begin{align*}
\myH^{(1)}_{\Omega, \Omega^c}(t)=
\myH^{(1)}_{\Omega \setminus \Omega_{\varepsilon},\Omega^{c}}(t)+ 
\myH^{(1)}_{\Omega_{\varepsilon}, \Omega^{\delta}}(t)+
\myH^{(1)}_{\Omega_{\varepsilon},\Omega^c\setminus \Omega^{\delta}}(t),
\end{align*}
so that in order to prove part $(ii)$ of Theorem \ref{2thm},
we still need  to show the following.

\begin{lem} Let $\Omega\subset \Rd$ be  a bounded domain and consider $\Omega_{\varepsilon}$ and $\Omega^{\delta}$ the
inner and outer tubular neighbourhoods of $\bd$ defined in
\eqref{inout}. Then,
\begin{align*}
\lim\limits_{\tgo}
\frac{\myH^{(1)}_{\Omega \setminus \Omega_{\varepsilon}, \Omega^c}(t)}{t\ln\pthesis{\frac{1}{t}}}
=\lim\limits_{\tgo}\frac{\myH^{(1)}_{\Omega_{\varepsilon},\Omega^c\setminus \Omega^{\delta}}(t)}{t\ln\pthesis{\frac{1}{t}}}=0.
\end{align*} 
\end{lem}
\begin{proof}
By appealing to the fact that $\Omega^c \subset B_{\md}^c(x) $  for every $x\in \Omega$  , we  observe that
\begin{align}
\myH_{\Omega \setminus \Omega_{\varepsilon},\Omega^c}^{(1)}(t)&=k_d\,\, t\,\, \int_{\Omega\setminus \Omega_{\varepsilon}}dx\int_{\Omega^c}
\frac{dy}{\pthesis{t^2+|x-y|^2}^{\frac{d+1}{2}}}\nonumber \\&\leq
k_d\,\, t\,\, \int_{\Omega\setminus \Omega_{\varepsilon}}dx\int_{\Omega^c}dy\,|x-y|^{-d-1}\nonumber \\
&\leq k_d\, t\,\int_{\Omega \setminus \Omega_{\varepsilon}}dx
\int_{B_{\md}^c(x)}dy\,\,|x-y|^{-d-1}\label{pc}\\
&=k_d\,t\,\mathcal{H}^{d-1}(\partial B_1(0))\,\int_{\Omega\setminus \Omega_{\varepsilon}}dx\,\, \md^{-1} \nonumber \\
&\leq k_d\,\mathcal{H}^{d-1}(\partial B_1(0))\, \varepsilon^{-1} \,|\Omega|\,t,\nonumber
\end{align}
where in the last inequality we have used that
$\Omega \setminus \Omega_{\varepsilon}=\set{x\in \Omega: \md \geq \varepsilon},$
whereas to compute the integral term in \eqref{pc} we have employed spherical coordinates.

Next, since $\bar{\Omega}$ is compact, 
we have 
\begin{align}\label{radiuso}
0<r_{\Omega}=\sup\limits_{x\in \Omega}|x|<\infty.
\end{align}
Choose any $r>r_{\Omega}$ and notice that $\Omega\subset B_r(0)$. Thus, we find that
\begin{align*}
\int_{\Omega_{\varepsilon}}dx
\int_{\Omega^c\setminus \Omega^{\delta}}dy\,|x-y|^{-d-1}&=
\int_{\Omega_{\varepsilon}}dx
\int_{\pthesis{\Omega^c \setminus \Omega^{\delta}}\cap B_{r}(0)}dy
\,|x-y|^{-d-1} \\
&+\int_{\Omega_{\varepsilon}}dx\int_{\pthesis{\Omega^c \setminus \Omega^{\delta}}\cap B_{r}^c(0)}dy\,|x-y|^{-d-1}.
\end{align*}
Note that for all $x\in \Omega$ and $y \in \Omega^c \setminus \Omega^{\delta}$, we have the following inequality
$\delta\leq \rho_{\Omega}(y)\leq |x-y|.$
 Thus, the first integral term on  the right hand side of the previous equality is  bounded above by
$\delta^{-d-1}|\Omega||B_r(0)|.$
As far as the second integral is concerned, we have  for all $x\in \Omega_{\varepsilon}$ and \eqref{radiuso} that
\begin{align*}
|y-x|\geq |y|-|x|\geq |y|-r_{\Omega}. 
\end{align*}
Thus, 
\begin{align*}
\int_{\Omega_{\varepsilon}}dx\int_{\pthesis{\Omega^c\setminus
\Omega^{\delta}}\cap B_{r}^c(0)}dy\,|x-y|^{-d-1}&\leq |\Omega|\int_{B_{r}^c(0)}dy\pthesis{|y|-r_{\Omega}}^{-d-1}
\end{align*}
for all $r>r_{\Omega}$. By appealing to polar coordinates and the binomial theorem, we obtain
\begin{align*}
\int_{B_{r}^c(0)}dy\pthesis{|y|-r_{\Omega}}^{-d-1}&=\mathcal{H}^{d-1}(\partial B_1(0))\int_{r-r_{\Omega}}^{\infty}dp\,
(p+r_{\Omega})^{d-1}\,p^{-d-1}\\ \nonumber
&=\mathcal{H}^{d-1}(\partial B_1(0))\,
\mysum{j}{0}{d-1}  \binom {d-1} {j}\frac{r_{\Omega}^{\,j}}{(j+1)(r-r_{\Omega})^{j+1}}<\infty.
\end{align*}
Hence, we have shown that
\begin{align*}
\myH_{\Omega_{\varepsilon},\Omega^{c}\setminus \Omega^{\delta}}^{(1)}(t)\leq
C_{\delta,\Omega,\varepsilon}t.
\end{align*}
Finally, the assertion of the Lemma follows by combining all the estimates given above.
\end{proof}

{\bf Proof of part $(iii)$ of Theorem \ref{2thm}}
As before, we notice that $\Omega^c \subset B_{\md}^c(x) $ for every $x\in \Omega$ so that
\begin{align}\label{1stest}
\int_{\Omega_{\varepsilon}}dx\,\int_{\Omega^c}|x-y|^{-d-\alpha}
&\leq \int_{\Omega_{\varepsilon}}dx\,
\int_{B_{\md}^c(0)}dz\,|z|^{-d-\alpha} \nonumber \\
&=\mathcal{H}^{d-1}(\partial B_1(0))\, \alpha^{-1}\, \int_{\Omega_{\varepsilon}}dx\, \rho_{\Omega}^{-\alpha}(x).
\end{align}

Since uniformly $C^{1,1}$  bounded domains are also $R$--smooth boundary domains,  we  have according to Corollary 2.14 in \cite{BanKul} that there exists $\varepsilon >0$ 
(this $\varepsilon$ might not be the same   provided in Proposition \ref{Cdomprop}, however we can choose the smaller of them) such that
\begin{align}\label{est1}
\mathcal{H}^{d-1}(\partial \Omega_r)\leq 2^{d-1}\mathcal{H}^{d-1}(\bd),
\end{align}
for all $0<r<\varepsilon$. Hence, by the co--area formula, we obtain
\begin{align}\label{upperalphaper}
\int_{\Omega_{\varepsilon}}dx\, \rho_{\Omega}^{-\alpha}(x)&=
\int_{0}^{\varepsilon}dr\,r^{-\alpha}\,\mathcal{H}^{d-1}(\partial \Omega_r)
\leq 2^{d-1}\,(1-\alpha)^{-1}\, \mathcal{H}^{d-1}(\bd)\,\varepsilon^{1-\alpha}<\infty.
\end{align}
Likewise, as in \eqref{1stest} and using that $\Omega\setminus\Omega_{\varepsilon}=\set{x\in \Omega:\rho_{\Omega}(x)>\varepsilon}$, we have
\begin{align}\label{2ndest}
\int_{\Omega\setminus\Omega_{\varepsilon}}dx\,
\int_{\Omega^c}\,dy\,|x-y|^{-d-\alpha}&\leq
\mathcal{H}^{d-1}(\partial B_1(0))\, \alpha^{-1}\,\int_{\Omega\setminus\Omega_{\varepsilon}}dx\, \rho_{\Omega}^{-\alpha}(x)\nonumber \\
&\leq \mathcal{H}^{d-1}(\partial B_1(0))\, \alpha^{-1}\,\abs{\Omega}\varepsilon^{-\alpha}.
\end{align}
We have shown with \eqref{1stest} and \eqref{2ndest}
that
$$\myP_{\alpha}(\Omega)=\int_{\Omega}dx\,
\int_{\Omega^c}\,dy |x-y|^{-d-\alpha}<\infty$$
provided that $0<\alpha<1$. Thus, by combining the finiteness of the last integral with \eqref{tcomp}  and \eqref{tlim}, we conclude part $(iii)$ of Theorem \ref{2thm} by an application of the Lebesgue Dominated Convergence Theorem.

\section{upper bounds in theorem \ref{mcor}}\label{sec: ub}
Let $\Omega$ be a bounded domain. Then, it is clear that for every $x\in \Omega$, we have
$\tau_{B_{\md}(x)}^{(\alpha)}\leq \tau_{\Omega}^{(\alpha)}$
which implies
\begin{align}\label{exittimesprob}
\Prob^{x}\pthesis{ \tau_{\Omega}^{(\alpha)}<t}\leq
\Prob^{x}\pthesis{ \tau_{B_{\md}(x)}^{(\alpha)}< t}=
\Prob\pthesis{ \tau_{B_{\md}(0)}^{(\alpha)}< t}
\end{align}
for all $t>0$. Therefore, we conclude that
\begin{align*}
Q_{\Omega}^{(\alpha)}(t)=
\int_{\Omega}\,dx\,\Prob^{x}\pthesis{ \tau_{\Omega}^{(\alpha)}\geq t}=\abs{\Omega}-\int_{\Omega}\,dx\,\Prob^{x}\pthesis{ \tau_{\Omega}^{(\alpha)}< t}
\end{align*}
satisfies for all $t>0$ the following inequality 
\begin{align}\label{upper1}
\abs{\Omega}-Q_{\Omega}^{(\alpha)}(t)\leq
\int_{\Omega}\,dx\,\Prob\pthesis{ \tau_{B_{\md}(0)}
^{(\alpha)}< t}.
\end{align}

We now turn to the following result  whose proof and applications to subordinate killed Brownian
motion in a domain can be found in \cite[Proposition 2.1]{Song2}.

\begin{prop}\label{Songresult}
Assume $D$ is a bounded domain  satisfying  an exterior cone condition. Then, there exists $C\in (0,1)$ such that
\begin{equation*}
C\, \mathbb{P}^x( \tau_{D}^{(2)}\leq S_t)\leq\mathbb{P}^x(\tau_{D}^{(\alpha)}\leq t)
\leq \mathbb{P}^x( \tau_{D}^{(2)}\leq S_t),
\end{equation*}
for all $t>0$ and $x\in D.$
\end{prop}
In particular, by appealing to the last proposition with
$D=B_{\md}(0)$ and 
\eqref{upper1}, we find that
\begin{align}\label{upper2}
\abs{\Omega}-Q_{\Omega}^{(\alpha)}(t)\leq
\int_{\Omega}\,dx\,\,\Prob\pthesis{ \tau_{B_{\md}(0)}
^{(2)}< S_t}.
\end{align}
Next,   the independence between the Brownian Motion ${\bf B}$ and $\alpha/2$-subordinator ${\bf S}$ as stated in the introduction yields 
\begin{align}\label{upper3}
\Prob\pthesis{ \tau_{B_{\md}(0)}
^{(2)}< S_t}&= \Prob\pthesis{ \tau_{B_{\md}(0)}
^{(2)}< t^{2/\alpha}S_1}\\ \nonumber
&=\int_{0}^{\infty}\,ds\,\eta_1^{(\alpha/2)}(s)\,\Prob_{\bf B}\pthesis{ \tau_{B_{\md}(0)}
^{(2)}< t^{2/\alpha}s}.
\end{align}
In \cite[Lemma 3.3]{van2},  it is shown that
\begin{align*}
\Prob_{\bf B}\pthesis{ \tau_{B_{\md}(0)}
^{(2)}< t^{2/\alpha}s}
\leq 2^{(d+2)/2}\exp 
\pthesis{- \frac{\rho_{\Omega}^2(x)}{8\, t^{2/\alpha}\,s}},
\end{align*} 
so that by using Fubini's theorem, \eqref{upper2} and \eqref{upper3},
we arrive at
\begin{align}\label{Shcineq}
\abs{\Omega}-Q_{\Omega}^{(\alpha)}(t)\leq
2^{(d+2)/2}\int_{\Omega}\,dx\,\E{\exp\pthesis{- \frac{\rho_{\Omega}^2(x)}{8\, t^{2/\alpha}\,S_1}}}.
\end{align}

We split the foregoing integral as follows
\begin{align}\label{splitint}
\int_{\Omega}\,dx\,\E{\exp\pthesis{- \frac{\rho_{\Omega}^2(x)}{8\, t^{2/\alpha}\,S_1}}}= I_{\alpha}(t) + II_{\alpha}(t)
\end{align}
with
\begin{align*}
I_{\alpha}(t)&=\int_{\Omega_{\varepsilon}}\,dx\, \E{\exp\pthesis{- \frac{\rho_{\Omega}^2(x)}{8\, t^{2/\alpha}\,S_1}}},\\
II_{\alpha}(t)&=\int_{\Omega\setminus \Omega_{\varepsilon}}dx\,\E{\exp\pthesis{- \frac{\rho_{\Omega}^2(x)}{8\, t^{2/\alpha}\,S_1}}}
\end{align*}
and observe by \eqref{basicupperest} with $\kappa=\md\,8^{-1/2}\, t^{-\inalp}$, we obtain 
\begin{align}\label{2int}
II_{\alpha}(t)\leq C_{\alpha}\, \abs{\Omega}8^{\alpha/2}\varepsilon^{-\alpha}\,t
\end{align}
and by \eqref{est1} and co--area formula, we also have
\begin{align}\label{upperboundlimsup}
I_{\alpha}(t)&=\int_{0}^{\varepsilon}dr\,
\E{\exp\pthesis{-\frac{r^2}{ 8\, t^{2/\alpha}\,S_1}}}\mathcal{H}^{d-1}(\partial \Omega_{r})\\
&\leq2^{d-1}\, \mathcal{H}^{d-1}(\bd)\int_{0}^{\varepsilon} 
dr\, \E{\exp\pthesis{-\frac{r^2}{ 8\, t^{2/\alpha}\,S_1}}}\nonumber \\
&=2^{(2d+1)/2}\,\mathcal{H}^{d-1}(\bd)\,t^{\inalp}\int_{0}^{
\varepsilon\,8^{-1/2}\,t^{-\inalp}} 
dw\, \E{\exp\pthesis{-\frac{w^2}{S_1}}}\nonumber
\end{align}
for all $0<\alpha<2$ and $t>0$. 

 In order to obtain upper bounds it suffices to deal with  the integral term on the right hand side  of   the above inequality. As before, we divide this into various cases according to $\alpha$. 
\\

{\bf Case $1 < \alpha<2$:} By appealing to the identity 
\begin{align*}
\int_{0}^{\infty}
dw\,\exp\pthesis{-\frac{w^2}{s}}=2^{-1}\,\pi^{1/2}\,s^{1/2}
\end{align*}
 Fubini's Theorem and \eqref{gen.expt.S1}, we arrive at
\begin{align*}
\int_{0}^{\infty}
dw\, \E{\exp\pthesis{-\frac{w^2}{S_1}}}=2^{-1}\,\pi^{1/2}\,
\E{S_{1}^{1/2}}=2^{-1}\Gamma\pthesis{1-\inalp}
\end{align*}
so that by \eqref{upperboundlimsup}
\begin{align*}
I_{\alpha}(t)\leq 2^{(2d-1)/2}\,\mathcal{H}^{d-1}\,(\bd)\,t^{\inalp}\,\Gamma\pthesis{1-\inalp}.
\end{align*}
By putting together the preceding estimates and the inequalities \eqref{Shcineq} and \eqref{2int}, we obtain for all $t>0$ that
\begin{align*}
\abs{\Omega}-Q_{\Omega}^{(\alpha)}(t)
\leq 2^{(d+2)/2}\pthesis{C_{\alpha}\,\abs{\Omega}8^{\alpha/2}
\varepsilon^{-\alpha}\,t+2^{(2d-1)/2}\,\mathcal{H}^{d-1}(\bd)\,\Gamma\pthesis{1-\inalp}\,t^{\inalp}}.
\end{align*}
It easily follows that
\begin{align*}
\varlimsup\limits_{\tgo}\frac{\abs{\Omega}-Q_{\Omega}^{(\alpha)}(t)}{t^{\inalp}}\leq\,
2^{(3d+1)/2}\,\,\Gamma\pthesis{1-\inalp}\,\mathcal{H}^{d-1}(\bd).
\end{align*}

{\bf Case $\alpha=1$:} The $1/2$--subordinator ${\bf S}$ can be expressed as the first hitting time for the standard one-dimensional  Brownian motion $\set{W_t}_{t\geq0}$. More precisely, 
$
S_{t}=\inf\left\{s>0:W_s=\frac{t}{ \sqrt{2}}\right\}. 
$
It is known (see \cite{appleb} for details) that its transition density is given by
\[
\eta_{t}^{(1/2)}(s)= \frac{t}{2\sqrt{\pi}}s^{-3/2}e^{-t^2/4s}.
\]

A simple computation yields 
\begin{align*}
\E{\exp\pthesis{-\frac{w^2}{S_1}}}&=\frac{1}{\sqrt{4w^2+1}},\\
\int_{0}^{\varepsilon 8^{-1/2}\,\,t^{-1}}\frac{dw}{\sqrt{4w^2+1}}
&=\frac{1}{2}\ln\pthesis{\frac{1}{t}}+\frac{1}{2}
\ln\pthesis{ \frac{\varepsilon}{ \sqrt{2}}+
\sqrt{\frac{\varepsilon^2}{2}+t^2}}.
\end{align*}
Therefore, \eqref{Shcineq}, \eqref{2int} and the previous calculations show that $\abs{\Omega}-Q_{\Omega}^{(1)}(t)$ is bounded  above by
\begin{align*}
 2^{(d+2)/2}\pthesis{C_{1}\,\abs{\Omega}\sqrt{8}
\varepsilon^{-1}\,t+2^{(2d-1)/2}\,\mathcal{H}^{d-1}(\bd)\,t
\pint{
\ln\pthesis{\frac{1}{t}}+
\ln\pthesis{ \frac{\varepsilon}{ \sqrt{2}}+
\sqrt{\frac{\varepsilon^2}{2}+t^2}}
}}
\end{align*}
which in turn  implies 
\begin{align*}
\varlimsup\limits_{\tgo}
\frac{ \abs{\Omega}-Q_{\Omega}^{(1)}(t)}{t\ln\pthesis{\frac{1}{t}}}\leq\,
2^{(3d+1)/2}\,\,\mathcal{H}^{d-1}(\bd).
\end{align*}
\bigskip

{\bf Case $0<\alpha<1$:} For this case, it is not necessary to employ expression \eqref{splitint}. Instead, we apply  again  the inequality \eqref{basicupperest} with $\kappa=\md\,8^{-1/2}\, t^{-\inalp}$ to obtain
\begin{align*}
\int_{\Omega}\,dx \,\E{\exp\pthesis{- \frac{\rho_{\Omega}^2(x)}{8\, t^{2/\alpha}\,S_1}}}
\leq 8^{\alpha/2}\,C_{\alpha}\,\pthesis{\int_{\Omega}\,dx \,\,\rho_{\Omega}^{-\alpha}(x)}\,\,t.
\end{align*}
 The integral term at the right hand side turns out to be finite because of \eqref{upperalphaper}
and the fact 
\begin{align*}
\int_{\Omega\setminus \Omega_{\varepsilon}}dx\,\, \rho_{\Omega}^{-\alpha}(x) \leq \,
\varepsilon^{-\alpha} \,|\Omega|.
\end{align*}
Therefore, we find that
\begin{align*}
 \varlimsup\limits_{\tgo} 
\frac{|\Omega|-Q_{\Omega}^{(\alpha)}(t)}{t}\leq 2^{(d+2+3\alpha)/2}\,C_{\alpha}
\,\int_{\Omega}\,dx\, \rho_{\Omega}^{-\alpha}(x).
\end{align*}

Assume now that $\Omega$ also satisfies a uniform exterior volume condition. That is, there exists $c>0$ such that for any $\sigma \in \bd$ and any $r>0$ we have $\abs{B_r(\sigma)\cap \Omega^c}\geq c\,r^d$. Then, we  claim that for $ x \in \Omega$  the following inequality holds.
\begin{align}\label{claim1}
\frac{c}{2^{d+\alpha}}\, \rho^{-\alpha}_{\Omega}(x)\leq  \int_{\Omega^c}\frac{dy}{|x-y|^{d+\alpha}}.
\end{align}  To see this, let $x\in \Omega$ and choose $\sigma_x \in \bd$ such that $\md=\abs{\sigma_x-x}.$
Thus, for any $y$ belonging to $B_{\md}(\sigma_x)\cap \Omega^c$, we obtain
$\abs{x-y}\leq |x-\sigma_x|+|\sigma_x-y|\leq 2\,\md.$
Therefore, it follows from the last inequality that
\begin{align*}
\int_{\Omega^c}\frac{dy}{|x-y|^{d+\alpha}}&\geq \int_{B_{\md}(\sigma_x)\cap \Omega^c}
\frac{dy}{|x-y|^{d+\alpha}}\\
&\geq\frac{1}{2^{d+\alpha}}\abs{B_{\md}(\sigma_x)\cap \Omega^c}\md^{-d-\alpha}\geq \frac{c}{2^{d+\alpha}}\, \rho^{-\alpha}_{\Omega}(x). 
\end{align*}
In other words,
for bounded domains $\Omega$ with smooth boundary and $0<\alpha<1$, the small time behavior of 
$t^{-1}\pthesis{\abs{\Omega}-Q_{\Omega}^{(\alpha)}(t)}$
and the fractional $\alpha$-perimeter $\myP_{\alpha}(\Omega)$ defined in \eqref{alpper} are related and this completes the proof of Theorem \ref{mcor}.
\qed
\bigskip

{\bf Acknowledgments}: I am grateful  to my Ph.D supervisor,  Professor Rodrigo Ba\~nuelos, for his valuable suggestions and time while preparing this paper. I wish to  thank the referee for the many useful comments and corrections.


\end{document}